\pdfoutput=1

\documentclass[11pt]{article}
\usepackage[utf8]{inputenc}

\usepackage{amssymb}
\usepackage{mathtools}
\usepackage[usenames,dvipsnames]{xcolor}
\definecolor{parcolor}{rgb}{0.59, 0.29, 0}
\definecolor{seccolor}{rgb}{0.59, 0.29, 0}
\definecolor{distinguished}{rgb}{0, 0, 0}

\usepackage{hyperref}
\hypersetup{colorlinks=true,linkcolor={Brown},citecolor={Brown},urlcolor={Brown}}
\usepackage{sectsty}
\sectionfont{\color{Brown}}
\subsectionfont{\color{Brown}}
\subsubsectionfont{\color{Brown}}
\paragraphfont{\color{Brown}}
\usepackage{tocloft}
\usepackage[compress]{cleveref}
\usepackage{amsthm}
\usepackage{thmtools, thm-restate}
\declaretheoremstyle[%
  spaceabove=.8\baselineskip,%
  spacebelow=.8\baselineskip,%
  headfont=\bfseries,%
  notefont=\normalfont,%
  bodyfont=\itshape,%
  postheadspace=.5em%
]{thms}

\declaretheoremstyle[%
  spaceabove=.8\baselineskip,%
  spacebelow=.8\baselineskip,%
  headfont=\bfseries,%
  notefont=\normalfont,%
  bodyfont=\normalfont,%
  postheadspace=.5em%
]{defn}

\numberwithin{equation}{section}
\theoremstyle{thms}
\newtheorem{thrm}[equation]{Theorem}
\newtheorem{cor}[equation]{Corollary}
\newtheorem{lem}[equation]{Lemma}
\newtheorem{prop}[equation]{Proposition}

\theoremstyle{defn}
\newtheorem{cns}[equation]{Construction}
\newtheorem{defn}[equation]{Definition}
\newtheorem*{defn*}{Definition}
\newtheorem{exa}[equation]{Example}
\newtheorem{exc}[equation]{Exercise}

\newtheorem{notn}[equation]{Convention}
\newtheorem*{notn*}{Notation}
\newtheorem{rmk}[equation]{Remark}

  \crefname{prop}{Proposition}{Propositions}
  \crefname{cor}{Corollary}{Corollaries}
  \crefname{lem}{Lemma}{Lemmata}
  \crefname{cns}{Construction}{Constructions}
  \crefname{defn}{Definition}{Definitions}
  \crefname{rmk}{Remark}{Remarks}
  \crefname{notn}{Convention}{Conventions}
  \crefname{exa}{Example}{Examples}
  \crefname{exc}{Exercise}{Exercises}
  \crefname{thrm}{Theorem}{Theorems}
  \crefname{question}{question}{questions}
  \crefname{thrmprime}{Theorem'}{Theorem'}
\usepackage[shortlabels]{enumitem}
  \newlist{axenum}{enumerate}{1} 
  \setlist[axenum]{label=(\Alph*)}
  \crefname{axenumi}{Axiom}{Axioms}
\setlist{itemsep=.1em}

\usepackage{bm}
\usepackage[full]{textcomp} 
\usepackage[p,osf]{fbb} 
\usepackage[scaled=.95,type1]{cabin} 
\usepackage[varqu,varl]{zi4}
\usepackage[T1]{fontenc} 
\usepackage[libertine]{newtxmath}
\usepackage[cal=boondoxo,bb=boondox,frak=boondox]{mathalfa}
\usepackage{amsmath}

\usepackage{tikz}
\usetikzlibrary{arrows,shapes,positioning,backgrounds,cd,decorations.pathmorphing}
\tikzset{-,>=stealth',shorten >=2pt,shorten <=2pt,
  main node/.style={circle,fill=blue!20,inner sep=2pt,font=\sffamily\tiny\bfseries},
  desc/.style={font=\sffamily\footnotesize, align=center}
}
\usepackage{subcaption}
\usepackage{textcomp}
\usepackage{tikz-cd}
\RequirePackage[margin=3.3cm]{geometry}
\usetikzlibrary{arrows,arrows.meta}
\tikzcdset{arrow style=tikz}

\usepackage{xspace}
\newcommand{\icat}{$\infty$\nobreakdash-category\xspace}
\newcommand{\icats}{$\infty$\nobreakdash-categories\xspace}
\newcommand{\icategorical}{$\infty$-categorical\xspace}

\newcommand{\subicat}{sub-$\infty$-category\xspace}

\newcommand{\iCat}{\textup{Cat}_{\infty}}
\newcommand{\iCatst}{\textup{Cat}_{\infty}^{\textup{st}}}
\newcommand{\miCatst}{\textup{Cat}_{\infty}^{\textup{st},\otimes}}

\newcommand{\PrL}{\textup{Pr}^{\textup{L}}}
\newcommand{\PrLst}{\textup{Pr}^{\textup{L}}_{\textup{st}}}
\newcommand{\PrRst}{\textup{Pr}^{\textup{R}}_{\textup{st}}}
\newcommand{\mPrL}{\textup{Pr}^{\textup{L},\otimes}}
\newcommand{\PrR}{\textup{Pr}^{\textup{R}}}

\newcommand{\mPrst}{\textup{Pr}_{\textup{st}}^{\textup{L},\otimes}}

\newcommand{\loccit}{\textsl{loc.\,cit}.\@\xspace}
\input cyracc.def
\DeclareFontFamily{U}{russian}{}
\DeclareFontShape{U}{russian}{m}{n}
        { <5><6> wncyr5
        <7><8><9> wncyr7
        <10><10.95><12><14.4><17.28><20.74><24.88> wncyr10 }{}
\DeclareSymbolFont{Russian}{U}{russian}{m}{n}
\DeclareSymbolFontAlphabet{\mathcyr}{Russian}
\makeatletter
\let\@math@cyr\mathcyr
\renewcommand{\mathcyr}[1]{\@math@cyr{\cyracc #1}}
\makeatother

\newcommand{\inv}{^{-1}}
\DeclareMathOperator{\Sh}{Sh}
\DeclareMathOperator{\D}{\mathsf{D}}
\DeclareMathOperator{\Db}{\mathsf{D}^{\textup{b}}}
\DeclareMathOperator{\DM}{\mathsf{DM}}
\DeclareMathOperator{\DMB}{\mathsf{DM}_{\mathcyr{B}}}
\DeclareMathOperator{\DMBgm}{\mathsf{DM}_{\mathcyr{B}}^{\textup{gm}}}
\DeclareMathOperator{\DMc}{\mathsf{DM}^{\textup{c}}}
\DeclareMathOperator{\DMgm}{\mathsf{DM}^{\textup{gm}}}
\DeclareMathOperator{\Dbc}{\mathsf{D}^{\textup{b}}_{\textup{c}}}
\DeclareMathOperator{\Dbh}{\mathsf{D}^{\textup{b}}_{\textup{h}}}
\DeclareMathOperator{\Dbrh}{\mathsf{D}^{\textup{b}}_{\textup{rh}}}
\DeclareMathOperator{\Hm}{H}
\DeclareMathOperator{\MHM}{\mathsf{MHM}}
\newcommand{\dR}{\mathsf{R}}
\newcommand{\dL}{\mathsf{L}}
\newcommand{\Hmc}{\mathrm{H}_{\mathrm{c}}}
\DeclareMathOperator{\Hom}{Hom}
\DeclareMathOperator{\Spec}{Spec}
\newcommand{\one}{\mathbb{1}}

\newcommand{\id}{\on{id}}
\newcommand{\into}{\hookrightarrow}

\newcommand{\calg}[1]{%
  \operatorname{CAlg}(#1)}

\usepackage{etoolbox}
\newcommand{\fun}[3][]{%
  \ifblank{#1}{
    \operatorname{Fun}(#2,#3)}
  {
    \operatorname{Fun}^{#1}(#2,#3)}}
\newcommand{\Mod}[2][]{%
  \ifblank{#1}{
    \operatorname{Mod}(#2)}
  {
    \operatorname{Mod}_{#1}(#2)}}

\newcommand{\map}[3][]{%
  \ifblank{#1}{
    \operatorname{Map}(#2,#3)}
  {
    \operatorname{Map}_{#1}(#2,#3)}}

\newcommand{\psh}[2][]{%
  \ifblank{#1}{
    \mathcal{P}(#2)}
  {
    \mathscr{P}_{#1}(#2)}}

\providecommand{\aone}{\ensuremath{\mathbb{A}^{1}}}
\providecommand{\pone}{\ensuremath{\mathbb{P}^{1}}}

\DeclareMathOperator{\Ind}{Ind}

\newcommand{\FF}{\mathbb{F}}
\newcommand{\FFq}{\mathbb{F}_{\!q}}

\newcommand{\Q}{\mathbb{Q}}
\newcommand{\CC}{\mathbb{C}}
\newcommand{\DD}{\mathbb{D}}

\newcommand{\Z}{\mathbb{Z}}

\newcommand{\affine}{\mathbb{A}}

\newcommand{\on}[1]{\operatorname{#1}}

\newcommand{\Cgm}{C^{\textup{gm}}}

\newcommand{\twist}[1]{\kern-3pt\left\{#1\right\}}
\newcommand{\twistns}[1]{\left\{#1\right\}}

\providecommand{\ihom}{\underline{\mathrm{Hom}}}

\providecommand{\iCat}[1]{\mathrm{Cat}_{\infty}^{#1}}
\providecommand{\iCatst}{\iCat{\mathrm{st}}}
\providecommand{\miCatst}{\mathrm{Cat}_{\infty}^{\mathrm{st},\otimes}}
\providecommand{\miCat}{\mathrm{Cat}_{\infty}^{\otimes}}

\providecommand{\mTri}{\mathsf{TrCat}^{\otimes}}

\newcommand{\CoSy}[1]{
  \operatorname{CoSy}_{#1}}

\providecommand{\cosy}[1]{\mathbf{CoSy}^{\textup{sm}}_{\mathrm{#1}}}
\providecommand{\CoSyPr}[1]{\mathbf{CoSy}^{\textup{Pr}}_{\mathrm{#1}}}

\DeclareMathOperator{\ho}{Ho}
\DeclareMathOperator{\SH}{\mathsf{SH}}
\DeclareMathOperator{\FSH}{\mathsf{FSH}}
\DeclareMathOperator{\RigSH}{\mathsf{RigSH}}

\DeclareMathOperator{\thom}{Th}
\providecommand{\Comp}[1][]{
  \mathop{\mathrm{Comp}}_{\ifblank{#1}{}{#1}}}

\newcommand{\pt}{*}

\newcommand{\FSch}[2][]{
  \ifblank{#1}{\operatorname{FSch}^{#2}}{\operatorname{FSch}_{#1}^{\text{#2}}}}
\newcommand{\RigSpc}[2][]{
  \ifblank{#1}{\operatorname{RigSpc}^{#2}}{\operatorname{RigSpc}_{#1}^{\text{#2}}}}
\newcommand{\Sch}[2][]{
  \ifblank{#1}{\operatorname{Sch}^{\text{#2}}}{\operatorname{Sch}_{#1}^{\text{#2}}}}

\newcommand{\isoto}{\overset{\sim}{\,\to\,}}

\let\xto\xrightarrow
\newcommand{\mc}[1]{\mathcal{#1}}
\newcommand{\cA}{\mc{A}}
\newcommand{\cB}{\mc{B}}
\newcommand{\cC}{\mc{C}}
\newcommand{\cD}{\mc{D}}

\newcommand{\cF}{\mc{F}}
\newcommand{\cG}{\mc{G}}
\newcommand{\cH}{\mc{H}}
\newcommand{\cM}{\mc{M}}
\newcommand{\cO}{\mc{O}}

\newcommand{\cX}{\mc{X}}

\newcommand{\resp}[1]{%
  \textup{(}resp.\ #1\textup{)}\xspace}

\providecommand{\unit}{\mathbb{1}}
\providecommand{\base}{B}

\DeclareFontFamily{U}{mathb}{\hyphenchar\font45}
\DeclareFontShape{U}{mathb}{m}{n}{
      <5> <6> <7> <8> <9> <10> gen * mathb
      <10.95> mathb10 <12> <14.4> <17.28> <20.74> <24.88> mathb12
      }{}
\DeclareSymbolFont{mathb}{U}{mathb}{m}{n}
\DeclareMathSymbol{\boxasterisk}  {2}{mathb}{"66}
\makeatletter
\def\namedlabel#1#2{\begingroup
  \def\@currentlabel{#2}%
  \textbf{#2}
   \label{#1}\endgroup
}
\makeatother
\newcommand{\bref}[1]{\textbf{\ref{#1}}}

\newcommand{\kw}{motivic homotopy theory, \texorpdfstring{$\affine^1$}{A¹}-homotopy, coefficient system, six operations, six-functor formalism}
\newcommand{\ttl}{An introduction to six-functor formalisms}
\newcommand{\aut}{Martin Gallauer}
\title{\ttl}
\AtEndPreamble{\hypersetup{
    pdfcreator    = {\LaTeX{}},
    pdfencoding   = auto,
    psdextra,
    pdfauthor     = {\aut},
    pdftitle      = {\ttl},
    pdfsubject    = {\ttl},
    pdfkeywords   = {\kw}}}
\begin{document}
\author{\aut}
\maketitle{}
\begin{abstract}
\noindent These are notes for a mini-course given at the summer school and conference \href{https://sites.google.com/view/summer-school-motivic/home}{The Six-Functor Formalism and Motivic Homotopy Theory} in Milan~9/2021.
  They provide an introduction to the formalism of Grothendieck's six operations in algebraic geometry and end with an excursion to rigid-analytic motives.

The notes do not correspond precisely to the lectures delivered but provide a more self-contained account for the benefit of the audience and others.
  No originality is claimed.
\setcounter{tocdepth}{2}
\end{abstract}

\setcounter{section}{-1}
\section{Introduction}
\label{sec:intro}
The main goal of these lectures is to touch upon the following questions regarding six-functor formalisms:
\begin{enumerate}
\item \textit{Why care about them?}
\item \textit{What are they?}
\item \textit{How to construct them?}
\end{enumerate}
Needless to say, our answers are far from complete.
(We will try to give references for the reader who wants to venture further but they will certainly not be exhaustive either.)
In short, they are:
\begin{enumerate}
\item \emph{Why?} If one cares about cohomology then one should care about the six operations because the latter \textbf{enhance} the former.
Grothendieck's relative point of view is baked into it which connects well with modern algebraic geometry.
And, finally, the formalism has proven highly successful in the last decades.
(This is particularly apparent in motivic homotopy theory, the other main topic of the summer school.
Unfortunately, we will treat this last point only briefly and leave much to the other talks.)
\item
\emph{What?}
First the confession: There will be no definition of six-functor formalisms in these lectures.
Just as `cohomology' is arguably not a precisely defined term and varies from context to context, we cannot expect its enhancement to admit a definition pleasing everyone.
Instead we will try to give a glimpse of the six functors in action, and we will describe a convenient and precise framework to think about them.
This framework consists of \textbf{coefficient systems} which encode a minimal set of structure and axioms one would like a six-functor formalism to enjoy.
\item
\emph{How?}
Given the power of the formalism it is unsurprising that all known examples required major efforts, often by many mathematicians, until they were established.
(And in several areas, a `complete' formalism is still very much work in progress.)
We will focus on arguably the most common and serious stumbling block, the construction of the \textbf{exceptional functoriality}.
In a slightly different direction, such difficulties can be circumvented altogether by constructing six-functor formalisms out of already-established ones.
And finally, we will discuss by way of illustration, an example (from rigid-analytic geometry) of a recent new addition to the list of six-functor formalisms.
\end{enumerate}

We assume that the reader is familiar with basic scheme theory and has seen derived categories before.
\Cref{sec:why} is written in the language of triangulated categories although the axioms are barely used.
In \Cref{sec:what,sec:how} we use the language of stable \icats but some help is provided and much of it can also be understood just at the level of underlying triangulated categories.
We imagine that the more exposition to the various cohomology theories for schemes (or for other geometric objects) the reader has had the easier it will be to follow the text.

\paragraph*{Acknowledgements}
I would like to thank the organizers of the summer school and conference \textit{The Six-Functor Formalism and Motivic Homotopy Theory} for the opportunity to talk about this subject and for their hard work in making the event a success.
I would also like to thank all the participants of the summer school for their input during and after the lectures.
Bastiaan Cnossen read a previous version of these notes and contributed many corrections and suggestions that greatly improved the document.

\tableofcontents

\section{Why?}
\label{sec:why}

Here we will try to motivate the study and development of six-functor formalisms.
The point of view we will try to convey is that
\begin{quote}
\textit{six-functor formalisms enhance cohomology.}
\end{quote}
Interspersedly, we will also make comments about the related question why six-functor formalisms arose historically in the first place although this is not our focus.
There is little rigorous mathematics to be found here - for that we ask the reader to wait until \Cref{sec:what,sec:how}.

\begin{rmk}

Another natural way to answer the question in the title would be to list applications of the theory and thereby argue for its importance.
We will not do that here and, in any case, compiling an even approximately complete list would seem a daunting task.
Indeed, the language and theory of six-functor formalisms permeates much of modern algebraic geometry and beyond, and has spawned entire fields of research.
The development of, for example, \'etale cohomology, perverse sheaves, or motivic homotopy theory is quite unthinkable in the absence of the six operations.

\end{rmk}
\subsection{A hierarchy of invariants}
\label{sec:hierarchy-invariants}
\begin{exa}
\label{exa:betti-numbers}
If you are studying a topological space $X$, a useful invariant to know about is the sequence of Betti numbers $b_n(X)$, the latter measuring the number of $n$-dimensional holes in~$X$.
Famously, Noether explained how these numbers are just shadows of the homology of~$X$, these being a sequence of abelian groups~$\Hm_n(X)$ measuring the difference between cycles and boundaries on~$X$.
Thus homology is a richer invariant than the Betti numbers since there is a way to go from the former to the latter but no way (in general) to reverse this process:
\[
\begin{tikzcd}
\Hm_n(X)
\ar[d, squiggly, shift right=1em, "\mathrm{rank}\ " left]
\\
b_n(X)
\ar[u, "/"{anchor=center,sloped}, squiggly, shift right=1em]
\end{tikzcd}
\]
\end{exa}

\begin{exa}
\label{exa:weil-conjectures}
Now imagine instead a variety~$X$ over a finite field~$k=\FFq$ (of cardinality~$q=p^r$, say).
If you are an arithmetic geometer, chances are you would like to know the number of rational points, that is, solutions to the polynomial equations defining~$X$, possibly over finite extensions of~$k$.
The $\zeta$-function of~$X$ conveniently packages this information:
\[
\zeta_X(T)=\exp
\left(
\sum_{n\geq 1}\frac{\#X(\FF_{\!q^n})}{n}T^n
\right)\in\Q[\![T]\!]
\]
Moreover, if $X$ is smooth and proper, Weil~\cite{MR29393} predicted that this function is very nicely behaved: It should be a rational function, satisfy a certain functional equation, and one should have tight control over the zeroes and poles.
(Weil also proved this for curves.)
He suggested that these desirable properties would follow from a well-behaved cohomology theory for varieties over finite fields, a suggestion which was eventually realized by the concerted effort of many mathematicians, including Grothendieck, Serre and Deligne:
The theory of $\ell$-adic cohomology was at least partly developed to settle the Weil conjectures.

Here then we find a similar situation as in topology in that cohomology groups are richer invariants than individual numbers and that a certain behaviour of the former implies a certain behaviour of the latter:
\[
\begin{tikzcd}
\Hm^{\bullet}(X_{\bar{k}};\Q_\ell)
\ar[d, squiggly, "\textup{traces of (iterated) Frobenii}\ " left]
\\
\zeta_X(T)
\end{tikzcd}
\]
What is of interest to us in this historical example is that the `good behaviour' of these cohomology groups $\Hm^n(X_{\bar{k}};\Q_\ell)$ was in turn deduced from properties of an even richer invariant, the $\ell$-adic constructible derived category:
\[
\begin{tikzcd}
\Dbc(X_{\bar{k}};\Q_\ell)
\ar[d, squiggly, "\hom\textup{-groups}\ " left]
&&
\textup{category-level invariant}
\\
\Hm^{\bullet}(X_{\bar{k}};\Q_\ell)
\ar[d, squiggly, "\textup{trace of Frobenius}\ " left]
&&
\textup{set-level invariant}
\\
\zeta_X(T)
&&
\textup{element-level invariant}
\end{tikzcd}
\]
Summarizing, in order to prove certain things about element-level invariants mathematicians in this case have found themselves proving things about category-level invariants two levels up and deducing the former from the latter.
\end{exa}

\begin{rmk}
\label{rmk:slogan-1}
Jumping ahead of ourselves for a moment, we can say that
\begin{quote}
\emph{six-functor formalisms govern the behaviour of certain category-level invariants.}
\end{quote}
Therefore we can expect that this formalism will be useful in proving things about certain set-level invariants, namely the cohomology of `spaces' (these could be topological spaces or spaces appearing in algebraic geometry and even further beyond).
To say something more precise we have to appreciate one important feature which arises in set-level and category-level invariants but is absent at the lower end of the hierarchy.
This we will try to do in \Cref{sec:relative-point-view}.
\end{rmk}

\begin{rmk}
The Weil conjectures (\Cref{exa:weil-conjectures}) are discussed from this point of view in an unpublished note of Voevodsky~\cite{voevodsky-state}.
In the same note (from the year~2000) he expresses the view that the development of the six-functor formalism would become one of the main technical tools in advancing motivic homotopy theory:
a view which over the last 20 years has certainly materialized!
\end{rmk}
\subsection{Relative point of view}
\label{sec:relative-point-view}

Grothendieck famously stressed the `relative point of view', replacing for example schemes by morphisms of schemes as the fundamental object of study.
This shift is also apparent in the development of $\ell$-adic cohomology and the proof of the Weil conjectures (\Cref{exa:weil-conjectures}).

\begin{rmk}
Even if one is ultimately interested in the cohomology of a single variety~$X$ it is often necessary to invoke other, related varieties and their cohomologies in the process, for example in arguments that proceed by induction on the dimension, or when covering~$X$ by simpler pieces.
It then becomes important to study the cohomology groups not in isolation but together with the maps
\begin{equation}
f^*:\Hm^{\bullet}(Y)\to \Hm^{\bullet}(Y')
\label{eq:action-f-cohomology}
\end{equation}
for all morphisms $f:Y'\to Y$.

And even if not passing through other varieties, the action of a morphism on cohomology provides additional, often very interesting information about the varieties involved.
In the discussion of the Weil conjectures (\Cref{exa:weil-conjectures}) we already saw an example of this phenomenon:
The action of the Frobenius endomorphism is used to express the number of rational points in terms of cohomology.
\end{rmk}

\begin{rmk}
\label{rmk:cohomology-compact-support}
For the same reason, even if one is interested in proper varieties it is sometimes necessary to invoke non-proper varieties and their cohomologies in the process.
The latter are typically less well-behaved and to make up for that, the notion of \emph{cohomology with compact support} was developed.
Thus in addition to cohomology groups we also want to study the groups $\Hmc^\bullet(X)$ and their dependence on~$X$.
\end{rmk}

\subsection{The six functors in topology}
\label{sec:six-funct-topology}

In \Cref{sec:hierarchy-invariants} we saw that moving up along the hierarchy of invariants, cohomology is replaced by sheaves, and in \Cref{sec:relative-point-view} we stressed the need to adopt a relative point of view.
Putting the two together one arrives at the study of assignments
\[
C:\text{spaces}\longrightarrow \text{categories}
\]
which send a space to some category of sheaves on that space, and where morphisms of spaces induce functors between the corresponding categories of `sheaves'.
(As we saw in \Cref{exa:weil-conjectures}, these are not necessarily literally sheaves but something related, such as derived categories of sheaves.)
The latter are examples of the functors giving `six-functor formalisms' their name.
Sometimes they are also called \emph{operations} since they operate on sheaves.

For Weil and Grothendieck, a good cohomology theory for varieties over finite fields was to behave similarly to the cohomology of topological spaces.
It is therefore prudent to look at the topological situation first.
This we do briefly here.
References that include much more detail include~\cite{MR842190,MR2050072,kashiwara-schapira-90}.

\begin{exa}
\label{exa:inverse-direct-image-topology}
Let us go back to the topological example~\ref{exa:betti-numbers}.
For a nice enough\footnote{for example, cohomologically locally connected in the sense of~\cite{petersen:sheaf-cohomology}} space~$X$ the cohomology $\Hm^{\bullet}(X)$ coincides with the sheaf cohomology of the constant sheaf on~$X$.
The most familiar operations on (abelian) sheaves associated with a continuous map $f:X\to Y$ are the \emph{inverse image} (or \emph{pull-back}) and \emph{direct image} (or \emph{push-forward}), respectively:
\[
\begin{tikzcd}
\Sh(Y)
\ar[r, shift left=.5em, "f^*"]
&
\Sh(X)
\ar[l, shift left=.5em, "f_*"]
\end{tikzcd}
\]
Recall that $f_*\cF$ is the sheaf whose sections on an open subset $U\subseteq Y$ are given by $\Gamma(f\inv(U),\cF)$.
The functor $f^*$ can be defined either as the left adjoint to $f_*$ or via a more or less explicit construction.

Note that if $Y=\pt$ is just a point, the functor $f_*:\Sh(X)\to \Sh(\pt)\simeq\Mod{\Z}$ coincides with the global sections functor.
We deduce that the right derived functor coincides with sheaf cohomology:
\begin{equation}
\dR^nf_*(\cF)\cong \Hm^n(X;\cF)\label{eq:push-forward-cohomology}
\end{equation}
Thus we may view the derived push-forward as a relative and enhanced version of cohomology.
\end{exa}

\begin{exa}
Continuing with \Cref{exa:inverse-direct-image-topology}, another familiar operation is the \emph{direct image with compact support} (or \emph{compactly supported push-forward}).
It is defined as a subfunctor of the direct image functor:
\[
\Gamma(U,f_!\cF):=\{s\in\Gamma(U,f_*\cF)=\Gamma(f\inv(U),\cF)\mid s\text{ has compact support}\}
\]
Note that again for $Y=\pt$ a point, the direct image with compact support recovers cohomology with compact support:
\begin{equation}
\dR^nf_!(\cF)\cong \Hmc^n(X;\cF)
\label{eq:push-forward-!-cohomology}
\end{equation}

\end{exa}

\begin{rmk}
The functor $f_!:\Sh(X)\to \Sh(Y)$ does not admit an adjoint in general.
This together with the fact that we are ultimately interested in derived functors leads us to consider derived categories of sheaves instead.
It turns out that at least for locally compact Hausdorff spaces, the functor acquires a right adjoint:
\[
\begin{tikzcd}
\D(\Sh(X))
\ar[r, shift left=.5em, "\dR f_!"]
&
\D(\Sh(Y))
\ar[l, shift left=.5em, "f^!"]
\end{tikzcd}
\]
The functor $f^!$ is called \emph{exceptional inverse image} (or \emph{pull-back}).
Accordingly, $f_!$ is sometimes also called \emph{exceptional direct image} (or \emph{push-forward}).\footnote{Another common name is \emph{$f$-upper-shriek} for $f^!$ and \emph{$f$-lower-shriek} for $f_!$.}
\end{rmk}

\begin{rmk}
\label{rmk:6-functors-topology}
Together with the tensor product and internal hom of sheaves we have collected all six functors:
\[
\begin{tikzcd}
(\otimes^{\dL},\dR\ihom)
\\[-15pt]
\D(\Sh(X))
\ar[d, "\dR f_*", bend right=50]
\ar[d, "\dR f_!", bend left=50, swap]
\\[30pt]
\D(\Sh(Y))
\ar[u, "\dL f^*", bend left=60]
\ar[u, "f^!", bend right=60, swap]
\end{tikzcd}
\]
It is customary to drop the symbols $\dR$ and $\dL$ for derived functors as the context usually makes clear when derived functors are intended.
\end{rmk}
\subsection{Enhancing cohomology: structure}
\label{sec:enhancing-cohomology-structure}

\begin{notn}
\label{rmk:6ff-abstract-setup}
Let us now abstract from the specific topological situation and instead assume that with each `space' (topological space, scheme, stack, \ldots{}) $X$ we are given a closed tensor triangulated category $(C(X),\otimes,\ihom)$ and with each morphism of spaces $f:X\to Y$ two adjunctions $f^*\dashv f_*$, $f_!\dashv f^!$ of exact functors:
\[
\begin{tikzcd}
(\otimes,\ihom)
\\[-15pt]
C(X)
\arrow[bend right=50]{d}[name=L, label={[label distance=-7pt]0:${\scriptsize f_*}$}]{}
\ar[bend left=50]{d}[name=R, label=left:$f_!$]{}
\\[30pt]
C(Y)
\arrow[bend left=60]{u}[label={[label distance=-17pt]0:${\scriptsize f^*}$}]{}
\arrow[bend right=60, swap]{u}[label={[label distance=-7pt]0:${\scriptsize f^!}$}]{}
\arrow[shorten <=18pt,shorten >=12pt,Rightarrow,to path={(R) -- node {} (L)}]{}
\end{tikzcd}
\]
The arrow $\Leftarrow$ indicates a natural transformation $f_!\Rightarrow f_*$ (which, in topology, is induced by the inclusion of sections with compact support).
We also assume that $f^*$ is endowed with a symmetric monoidal structure.
As a first approximation, the category $C(X)$ may be thought of as a `derived category of sheaves on~$X$' although we don't want to assume that this is literally the case.
To have more neutral language we will refer to objects in $C(X)$ as \emph{coefficients}, just as one speaks of cohomology with coefficients.

While we won't stress this aspect, it is important that the dependence on $X$ and $f$ is `pseudo-functorial'.
For example, we should in addition be given natural isomorphisms $g^*f^*\cong (fg)^*$ satisfying familiar cocycle conditions, and $\id^*\cong \id$.
For a precise formulation we refer to~\cite[Definition~2.2]{voevodsky:cross-functors}.
\end{notn}

\begin{rmk}
  We will now discuss how this basic setup allows us to recover structure present in cohomology.
In \Cref{sec:enhancing-cohomology-properties} we will see some properties of the six functors and how these properties govern the behaviour of cohomology.
\end{rmk}

\begin{rmk}
\label{rmk:cohomology-abstract}
The identifications in~(\ref{eq:push-forward-cohomology}) and~(\ref{eq:push-forward-!-cohomology}) show how the sheaf operations allow us to recover cohomology of spaces.
In the basic setup of \Cref{rmk:6ff-abstract-setup} we may take this as our \emph{definition}.
Let $p:X\to B$ be a morphism of spaces, where we think of $B$ as a `base space', fixed by the context.
For any coefficient $\cF\in C(X)$, the \emph{cohomology} \resp{\emph{with compact support}} is
\[
\Hm^\bullet(X;\cF):= p_*\cF\in C(B) \qquad (\textup{resp }\Hmc^\bullet(X;\cF):= p_!\cF\in C(B)).
\]
When $\cF=\unit$ is the tensor unit we denote these coefficients simply by $\Hm^\bullet(X)$ and $\Hmc^\bullet(X)$, respectively.

In order to obtain actual cohomology groups one may take appropriate homomorphism groups:
\[
\Hm^n(X;\cF):= \hom_{C(B)}(\unit,p_*\cF[n]) \qquad (\textup{resp }\Hmc^n(X;\cF):= \hom_{C(B)}(\unit,p_!\cF[n]))
\]
\end{rmk}

\begin{rmk}
\label{rmk:homology-abstract}
One can also define \emph{homology} and \emph{Borel-Moore homology}, generalizing these theories from topology, like so:
\begin{center}
\begin{tabular}{ccrcl}
  cohomology && $p_*p^*\unit$&&$\Hm^\bullet$\\
  cohomology with compact support && $p_!p^*\unit$&&$\Hmc^\bullet$\\
  homology && $p_!p^!\unit$&&$\Hm_\bullet$\\
  Borel-Moore homology && $p_*p^!\unit$&&$\Hm_\bullet^{\textup{BM}}$
\end{tabular}
\end{center}
\end{rmk}

\begin{exa}
\label{exa:6ffs}
Let $k$ be a field in which the prime $\ell$ is invertible and such that $\mathrm{cd}_\ell(k)<\infty$.
Then one has a structure as in \Cref{rmk:6ff-abstract-setup} which sends each finite type $k$-scheme (or even algebraic stack)~$X$ to the $\ell$-adic constructible derived category $\Dbc(X;\Q_\ell)$ (see e.g.~\cite{MR2434693}, although much of it goes back to SGA, particularly~\cite{MR0354654,deligne:sga4.5}).
In this case the cohomology \resp{with compact support} as defined in \Cref{rmk:cohomology-abstract} recovers $\ell$-adic cohomomology \resp{with compact support}.

Here are some more examples:\footnote{Some of these are only partial examples, in others certain technical assumptions are required.} 
\begin{center}
\begin{tabular}[c]{c|c}
  Coefficients&cohomology groups\\\hline
  $\Dbc(X;\Q_\ell)$ \emph{constructible $\ell$-adic sheaves}&
                                                              \emph{$\ell$-adic cohomology}\\
  $\Dbc(X(\CC);\Z)$ \emph{constructible analytic sheaves}&
                                                           \emph{Betti cohomology}\\
  $\Dbh(\cD_X)$ \emph{holonomic }$\mathcal{D}$\emph{-modules}&
                                                               \emph{de\,Rham cohomology}\\
  $\Db(\mathrm{Coh}(X))$ \emph{coherent sheaves}&
                                                  \emph{coherent cohomology}\\
  $\Db(\MHM(X))$ \emph{mixed Hodge modules}&
                                             \emph{absolute Hodge cohomology}\\
  $\DM(X)$ \emph{Voevodsky motivic sheaves}&
                                             \emph{(weight-$0$) motivic cohomology}\\
  $\SH(X)$ \emph{stable motivic homotopy sheaves}&
                                                   \emph{stable motivic (weight-$0$) cohomotopy groups}
\end{tabular}
\end{center}
\end{exa}

\begin{rmk}
Consider now a relative situation
\[
\begin{tikzcd}
X
\ar[rr, "f"]
\ar[rd, "p", swap]
&&
Y
\ar[ld, "q"]
\\
&
\pt
\end{tikzcd}
\]
The unit of the adjunction $ f^*\dashv f_*$ induces a morphism
\[
\eta:q_*\to q_*f_*f^*\cong p_*f^*
\]
and thus a morphism in cohomology
\[
\Hm^\bullet(Y,\cF)\to\Hm^\bullet(X,f^*\cF).
\]
If $\cF=\unit_X$ one recovers the action of $f$ on the cohomology of~$X$ as in~(\ref{eq:action-f-cohomology}).
\end{rmk}

\begin{rmk}
With compactly supported cohomology the situation is more subtle.
In the topological context, a natural map
\[
\Gamma_{\textup{c}}(Y,F)\to\Gamma_{\textup{c}}(X,f^*F)
\]
is defined when $f:X\to Y$ is proper.
Namely, in that case pulling back sections restricts to those with compact support.
This map is in turn induced by the same unit of the adjunction,
\[
\eta:q_!\to q_!f_*f^*\cong p_!f^*,
\]
using that $f_!=f_*$ as $f$ is proper:
\[
\Hmc^n(Y,\cF)\to \Hmc^n(X,f^*\cF)
\]
Similar functoriality exists for proper morphisms of schemes and other `spaces'.
\end{rmk}

\begin{exc}
Describe the functoriality of homology.
\end{exc}

\begin{rmk}
Let $p:X\to B$ be a morphism and recall that the inverse image $p^*:C(B)\to C(X)$ is symmetric monoidal.
It follows formally that its right adjoint $p_*:C(X)\to C(B)$ sends commutative algebras to commutative algebras.
In particular, the cohomology $\Hm^\bullet(X)=p_*\unit\in C(B)$ has the structure of a commutative algebra.
We may view this as an enhancement of the \emph{cup product} in cohomology.

Indeed, evaluating the multiplication through appropriate hom-groups $\hom_{C(B)}(\unit,-[n])$ one obtains a cup product at the level of cohomology groups:
\[
\cup:\Hm^a(X)\times\Hm^b(X)\to \Hm^{a+b}(X)
\]
\end{rmk}

\begin{rmk}
\label{rmk:regulator}
Needless to say, this section does not exhaust all structures of interest in cohomology.
For example, \emph{vanishing} and \emph{nearby cycles} are additional concepts of interest.
Another example will play a more important role in \Cref{sec:what}.
The classical theorem of de\,Rham and its algebraic geometry version of Grothendieck identifies cohomology groups associated with different theories (de\,Rham and singular cohomology).
Relatedly, \emph{Chern classes} and \emph{regulator maps} may be seen as morphisms from certain cohomology groups of one theory to those of another.
We want to think of these as underlying `(iso)morphisms of six-functor formalisms'.
For example, the Beilinson regulator maps algebraic K-theory classes to absolute Hodge cohomology, and this should arise  from a family of \emph{Hodge realization functors}
\[
\rho^*_{\textup{H}}(X):\DMc(X)\to \Db(\MHM(X))
\]
from categories of (constructible) motivic sheaves that `realize' the underlying Hodge cohomology of motives.
Moreover, ideally we would like these functors to be suitably compatible with the six operations.\footnote{This particular example is taken up again in \Cref{exa:hodge-modules}.}

\end{rmk}
\subsection{Enhancing cohomology: properties}
\label{sec:enhancing-cohomology-properties}
We now turn to properties of the six functors and related properties in cohomology.
We will discuss here only some of the many possibilities.
Our selection is geared towards the approach to six-functor formalisms described in \Cref{sec:what}.
\begin{rmk}
To avoid a possible confusion let us stress: The point is not (at least: not always) that results about a given cohomology theory come for free using six-functor formalisms.
But the difficulty can sometimes be shifted from establishing them directly to establishing that the cohomology theory underlies a six-functor formalism.
We will return to this in \Cref{sec:what,sec:how}.
\end{rmk}

\subsubsection{Proper push-forward}
\label{sec:ppf}

We already mentioned in the topological context that $f_!=f_*$ whenever $f$ is proper.
The same is true for six-functor formalisms in general:
Whenever $f$ is `proper' (for example, a proper morphism of schemes), the transformation $f_!\to f_*$ is an isomorphism.

\subsubsection{Duality} 
\label{sec:duality}
An important impetus for developing the six-functor formalism was what is now sometimes called \emph{Grothendieck duality}, as for example in~\cite{hartshorne:duality,MR0130879}.
In a very limited sense, in our setup this can be viewed as the computation of $f^!\unit$ for $f:X\to\pt$ smooth.
\begin{exa}[Topology]
\label{exa:dualizing-object-topology}
Let $X$ be a smooth manifold of dimension~$d$ and let $f:X\to \pt$ be the unique map.
For a ring $\Lambda$, one finds that $f^!\Lambda=\omega_{X,\Lambda}[d]$ is the (shifted) $\Lambda$-orientation sheaf.
Thus $X$ is orientable if and only if $\omega_{X,\Z}$ is the constant sheaf with value~$\Z$.
In that case, $\omega_{X,\Lambda}$ is constant for every ring~$\Lambda$.
\end{exa}

\begin{exa}[Coherent]
\label{exa:dualizing-object-coherent}
Let $X$ be a smooth $k$-variety of dimension~$d$.
If $f:X\to\Spec(k)$ denotes the structure morphism then $f^!k\cong \omega_{X}[d]$ is the (shifted) canonical sheaf on~$X$.
\end{exa}

\begin{exa}[$\ell$-adic]
\label{exa:dualizing-object-etale}
Let $X$ be a smooth $k$-variety of dimension~$d$ and $\ell$ a prime invertible in~$k$.
Then $f^!\Q_\ell\cong \Q_{\ell}(d)[2d]$ where $(d)$ denotes the $d$th Tate twist.
\end{exa}

\begin{cor}
\label{sta:PD-SD}
With the assumptions of \Cref{exa:dualizing-object-topology,exa:dualizing-object-coherent,exa:dualizing-object-etale} respectively, one has
\begin{enumerate}[(a)]
\item Poincar\'e duality (topology): if $X$ is orientable, $\Hm^n_c(X;\Q)^*\cong\Hm^{d-n}(X;\Q)$
\item Poincar\'e duality ($\ell$-adic): $\Hm^n_c(X;\Q_\ell)^*\cong\Hm^{2d-n}(X;\Q(d))$
\item Serre duality: if $X$ is proper, $\Hm^n(X;\cO_X)^*\cong\Hm^{d-n}(X;\omega_X)$
\end{enumerate}
\end{cor}
\begin{proof}
This follows from the adjunction isomorphisms
\[
\Hm_c^n(X)^*\cong\hom(f_!f^*\unit[n],\unit)\cong\hom(\unit,f_*f^!\unit[-n])
\]
together with the computations reported in \Cref{exa:dualizing-object-topology,exa:dualizing-object-coherent,exa:dualizing-object-etale}.
\end{proof}

\begin{rmk}
The coefficient $f^!\unit$ tries to be a dualizing object.
Verdier duality is concerned with the functors $\DD=\ihom(-,f^!\unit)$ and asks under which conditions one has isomorphisms such as
\[
\id\isoto\DD\circ\DD, \qquad \DD g_!\isoto g_*\DD, \qquad g^*\DD\isoto \DD g^!.
\]
It provides a relative version and generalization of duality phenomena such as the ones of \Cref{sta:PD-SD}.
\end{rmk}

\begin{exc}[Atiyah duality]
\label{exc:atiyah}
Let $f:X\to B$ be a smooth and proper morphism.
Show that the coefficient $\Hm^\bullet(X)=\Hmc^\bullet(X)$ is \emph{rigid}, with $\otimes$-dual given by $\Hm_\bullet(X)=\Hm_\bullet^{\textup{BM}}(X)$.\footnote{We use `rigid' instead of `strongly dualizable'. Recall that an object $a$ in a symmetric monoidal category is rigid if there is an object $a^*$ (called its $\otimes$-\emph{dual}) and morphisms $\unit\to a\otimes a^*$ and $a^*\otimes a\to \unit$ satisfying the identities familiar from adjunctions, see~\cite{dold-puppe:trace}.}

You will want to use the following two fundamental properties:
\begin{enumerate}[(a)]
\item \namedlabel{proj}{(Proper projection formula)}
for arbitrary $f$:
\[
f_!(f^*F\otimes G)\isoto F\otimes f_!G
\]
\item if $f$ is smooth then
\[
f^*F\otimes f^!H\isoto f^!(F\otimes H)
\]
\end{enumerate}
Note that the two morphisms are related by adjunction.
The second one is a form of relative purity, to which we now turn.
\end{exc}

\subsubsection{Relative purity}
\begin{rmk}
\label{rmk:relative-purity}  
Previously we `computed' $f^!\unit$ in the case where $f:X\to \pt$ is smooth.
It is natural to want to generalize this to arbitrary \emph{smooth} morphisms $f:X\to Y$,\footnote{In the topological context, this should be interpreted as a topological submersion~\cite[Definition~3.3.1]{kashiwara-schapira-90}.}
and this is provided by \emph{relative purity}.
It implies
\begin{enumerate}[(1)]
\item that the difference between $f^!$ and $f^*$ is measured yet
  again by $f^!\one$:
\[
f^!\one\otimes f^*F\isoto f^!F
\]
(This is equivalent to the second hint in \Cref{exc:atiyah}.)
\item that the coefficient $f^!\one$ is $\otimes$-invertible.
\end{enumerate}
Moreover, the coefficient $f^!\one$ arises from the \emph{Thom construction} (see below) applied to the relative tangent bundle~$T_f$, and this information is very useful in computations.
We interpret the equivalence $f^!\one\otimes(-)=:\twistns{T_f}$ as a `twist' by the relative tangent bundle and may therefore rewrite:
\begin{equation}\label{eq:relative-purity-twists}
\twistns{T_f}f^*\simeq f^!
\end{equation}
\end{rmk}

\begin{exa}
\label{exa:thom-ladic}
In the $\ell$-adic setting the Thom construction depends only on the rank of the vector bundle and the relative purity isomorphism reads as (see \Cref{exa:tate} below)
\begin{equation}
  \label{eq:relative-purity-ladic}
  f^*(d)[2d]\simeq f^!.\footnote{This is more generally true for \emph{orientable} theories (that is, those with a good notion of Chern classes).
  Implicitly, we also used the canonical isomorphism $f^*\twist{d}\simeq\twistns{d}f^*$ which, in the notation introduced in \Cref{rmk:thom-construction}, is a particular instance of $f^*\twist{V}\simeq\twistns{f^{-1}V}f^*$ for any vector bundle $V$ on~$Y$.}
\end{equation}
Note how~(\ref{eq:relative-purity-ladic}) generalizes Poincar\'e duality in $\ell$-adic cohomology discussed in \Cref{sec:duality}.

One often abbreviates the operation $(d)[2d]$ by $\twistns{d}$ and this is our inspiration for the notation in~(\ref{eq:relative-purity-twists}).
\end{exa}

Before discussing the Thom construction, let us note some important consequences of relative purity.

\begin{rmk}
As a consequence we note that for $f$ smooth, the inverse image functor $f^*$ admits a left adjoint
\[
f_\sharp=f_!\twist{T_f}\quad\dashv\quad f^*.
\]
It satisfies the \namedlabel{smooth-proj}{(Smooth projection formula)}:
\[
f_\sharp(f^*\cF\otimes\cG)\isoto \cF\otimes f_\sharp\cG,
\]
which arises, by adjunction, from the composite
\[
f^*\cF\otimes\cG\to f^*\cF\otimes f^*f_\sharp\cG\isoto f^*(\cF\otimes f_\sharp\cG)
\]
of the unit of the adjunction $f_\sharp\dashv f^*$ and the monoidality of~$f^*$.
\end{rmk}

\begin{exa}
\label{exa:motive-smooth}
Note that for $f:X\to B$ smooth, the homology of $X$ may be expressed alternatively as
\[
\Hm_\bullet(X)=f_!f^!\unit\simeq f_\sharp f^*\unit.
\]
\end{exa}

\begin{rmk}
\label{rmk:thom-construction}
In terms of this left adjoint we can describe the Thom construction as follows.
Let $p:V\to X$ be a vector bundle with zero section $s:X\into V$.
The $V$-twist is defined as
\[
\twistns{V}:= p_\sharp s_*: C(X)\to C(X).
\]
Then the \emph{Thom construction} applied to~$V$ is defined as the evaluation of this functor at the unit:
\[
\thom(V)=\unit\twist{V}=p_\sharp s_*\unit
\]
\end{rmk}

\begin{exa}
\label{exa:etale-purity}
As mentioned in \Cref{rmk:relative-purity}, for $f$ smooth we have
\[
  f^!\one\simeq\thom(T_f).
\]
If $f$ is \'etale then the relative tangent bundle $T_f=X$ is of rank zero, $\thom(T_f)\simeq\one$, and one deduces that $f^!\simeq f^*$.
\end{exa}

\begin{exc}
Explain as a consequence that Borel-Moore homology is contravariantly functorial with respect to \'etale morphisms.
\end{exc}

\begin{rmk}
\label{rmk:thom-ktheory}
  The Thom construction yields a morphism from the monoid (with respect to direct sums) of isomorphism classes of vector bundles on~$X$ to the Picard group of~$C(X)$,
which can be extended to virtual vector bundles~\cite{MR902592,MR2651359}, perfect complexes and passes to the level of $K$-theory.
In particular, there is a group homomorphism
\[
K_0(X)\xto{\thom(-)}\mathrm{Pic}(C(X)).
\]

For a modern approach to this construction see for example~\cite[\S\,16]{bachmann-hoyois:norms}, and for a more general discussion of purity we refer to~\cite{MR3874952,Deglise-Jin-Khan:fund-classes}.
\end{rmk}

\begin{rmk}
Continue with the set-up of \Cref{rmk:thom-construction} and denote by $j:V\backslash X\into V$ the inclusion of the complement of the zero section.
  From Localization discussed just below in \Cref{sec:loc} we deduce an exact triangle
  \begin{equation}
    \label{eq:thom-triangle}
p_\sharp j_\sharp j^*p^*\unit\to p_\sharp p^*\unit\to p_\sharp s_*s^*p^*\unit
\end{equation}
in $C(X)$ which exhibits $\thom(V)$ as the cone of the canonical morphism
\[
\Hm_\bullet(V\backslash X)\to \Hm_\bullet(V).
\]
These are the analogues of the sphere and disk bundle associated with~$V$ in topology, respectively, and this justifies labeling the construction `Thom construction'.
\end{rmk}

Locally every vector bundle is trivial so we better understand these first.
By \Cref{rmk:thom-ktheory}, it is enough to understand the rank-$1$ case.
\begin{exa}
\label{exa:tate}
Let $V=\aone_X$ be the trivial vector bundle of rank~$1$ on $X$.
A fundamental property of six-functor formalisms in algebraic geometry is the contractibility of the affine line:
\begin{description}
\item[\namedlabel{hty}{($\aone$-homotopy)}] $p_\sharp p^*\isoto\id$
\end{description}
which implies that $\Hm_\bullet(\aone_X)=\Hm_\bullet(X)$.
By \Cref{exc:tate} below, the exact triangle~(\ref{eq:thom-triangle}) in $C(X)$ becomes
\[
\unit\oplus\unit\twist{1}\kern-3pt[-1]\to\unit\to\thom(\aone_X)
\]
so that $\thom(\aone_X)=\unit\twist{1}$.
\end{exa}

\begin{exc}
\label{exc:tate}
Show that the homology of $\mathbb{G}_m=\aone\backslash 0$ splits canonically:
\[
\Hm_\bullet(\mathbb{G}_m)=\unit\oplus\tilde{\Hm}_\bullet(\mathbb{G}_m)
\]
for some coefficient $\tilde{\Hm}_\bullet(\mathbb{G}_m)$ (which we think of as the \emph{reduced} homology of $\mathbb{G}_m$).
We define the \emph{Tate twists} (and shifts thereof)
\begin{align*}
\unit(1):=\tilde{\Hm}_\bullet(\mathbb{G}_m)[-1], && \unit\twist{1}:=\tilde{\Hm}_\bullet(\mathbb{G}_m)[1].
\end{align*}
Using that open covers give rise to Mayer-Vietoris exact triangles,\footnote{This is a particular instance of \Cref{exc:nisnevich-triangle} below.} show also that
\[
\Hm_\bullet(\pone)=\unit\oplus\unit\twist{1}.
\]
\end{exc}

\begin{rmk}
We deduce from the preceding discussion that 
\begin{description}
\item[\namedlabel{tate}{(Tate stability)}] $\unit\twist{1}=\thom(\aone)$ is $\otimes$-invertible.
\end{description}
We may then define, for any coefficient $F$ and $n\in\Z$, $F\twist{n}=F\otimes\unit\twist{n}$.
\end{rmk}

\begin{rmk}
  As soon as we have discussed smooth base change (\Cref{rmk:sbc}) we could establish that the Thom construction is Zariski local.
  Together with the localization property we are about to discuss (\Cref{rmk:localization}), \bref{tate} therefore is seen to imply the $\otimes$-invertibility of all Thom coefficients.
In the same vein, \bref{hty} is enough to imply the contractibility of any vector bundle.
\end{rmk}
\subsubsection{Localization}
\label{sec:loc}
Let $i:Z\into X$ be a closed immersion with open complement $j:X\backslash Z\into X$.
\begin{notn}
One defines the \emph{cohomology of $X$ with support in $Z$} to be
\[
\Hm_Z^\bullet(X):=i_!i^!\unit\in C(X)
\]
\end{notn}
(In other words, it is the homology of~$Z$ relative to~$X$.)
\begin{exa}
Applying $\hom_{C(X)}(\unit,-[n])$ this recovers the corresponding notion in topology and in $\ell$-adic cohomology.
In the coherent context, these groups may be better known under the name of \emph{local cohomology (with respect to~$Z$)}.
\end{exa}

\begin{rmk}
There is a so-called \emph{localization triangle} of functors $C(X)\to C(X)$:
\[
i_!i^!\to \id\to j_*j^*
\]
which we may apply to the tensor unit~$\one$ to obtain
\[
\Hm^\bullet_Z(X)\to \Hm^\bullet(X)\to\Hm^\bullet(X\backslash Z).
\]
The associated long exact sequence is a well-known cohomological tool:
\[
\cdots\to \Hm^n_Z(X)\to\Hm^n(X)\to\Hm^n(X\backslash Z)\to\Hm^{n+1}_Z(X)\to\cdots
\]
\end{rmk}

\begin{rmk}
\label{rmk:localization}
The localization triangle is in turn a consequence of the localization property of six-functor formalisms:
\begin{description}
\item[\namedlabel{loc}{(Localization)}] The sequence of triangulated categories
\[
C(Z)\xto{i_*}C(X)\xto{j^*}C(U)
\]
is a localization sequence.
\end{description}
This means that one is in the situation of a recollement~\cite[\S\,1.4]{MR751966}, and in particular that
\begin{enumerate}[(1)]
\item $i_*\simeq i_!$, $j_*$, $j_!$ are fully faithful,
\item the composites $j^*i_*$, $i^!j_*$ and $i^*j_!$ all vanish,
\item the pairs $(i^*,j^*)$ and $(i^!,j^!)$ are each conservative,
\item one has another localization sequence $j_!j^!\to \id\to i_*i^*\to^+$.
\end{enumerate}
\end{rmk}

\begin{rmk}
The last triangle may also be written as (with base space~$X$)
\[
\Hmc^\bullet(X\backslash Z)\to \Hmc^\bullet(X)\to\Hmc^\bullet(Z)
\]
and gives rise to the usual long exact sequence of pairs in compactly supported cohomology.
This follows from the identifications $j^*=j^!$ (\Cref{exa:etale-purity}) and $i_*=i_!$ (\Cref{sec:ppf}).
\end{rmk}

\begin{exa}
\label{exa:empty-reduced}
\begin{enumerate}
\item
  \label{it:empty}
  It follows from \bref{loc} that $C(\emptyset)\simeq 0$ (take $i=\id:\emptyset\to\emptyset$).
\item 
  Now let $X_{\textup{red}}$ be~$X$ with the reduced scheme structure, and $i:X_{\textup{red}}\into X$ the
  obvious closed immersion.
  It follows from \bref{loc} together with~\ref{it:empty} that $i_*:C(X_{\textup{red}})\isoto C(X)$ is an equivalence.
  In other words, six-functor formalisms are insensitive to nilpotents.
\end{enumerate}
\end{exa}

\begin{rmk}
In the $\ell$-adic setting, localization is an easy property.
In other contexts, however, it can be a substantial theorem.
For example, for $\mathcal{D}$-modules fully faithfulness of $i_*$ is known as \emph{Kashiwara's lemma}.
In motivic homotopy theory, the localization sequence is a fundamental result of Morel-Voevodsky called the \emph{Glueing theorem}.
\end{rmk}

\subsubsection{Blow-up}
\label{sec:blowup}
The relation between the cohomology of a variety $X$ and its blow-up $\tilde{X}=\mathrm{Bl}_Z(X)$ is as simple as one might hope but it encodes a fundamental property of six-functor formalisms, namely proper base change.
\begin{notn}
We place ourselves in a more generally situation, with a commutative diagram of the following shape:
\begin{equation}
\label{eq:abstract-blow-up}
\begin{tikzcd}
\tilde{Z}
\ar[r, "\tilde{i}"]
\ar[rd, "w"]
\ar[d, "p'"]
&
\tilde{X}
\ar[d, "p"]
&
\tilde{X}\backslash\tilde{Z}
\ar[l, "\tilde{j}"]
\ar[d, "="]
\\
Z
\ar[r, "i"]
&
X
&
X\backslash Z
\ar[l, "j"]
\end{tikzcd}
\end{equation}
We assume that $p$ is proper and $i$ a closed immersion, both squares are Cartesian and the right vertical arrow is an isomorphism.
In this situation, the left part of the diagram is called an \emph{abstract blow-up square}.
\end{notn}

We now want to explain why there is an exact triangle (the \emph{blow-up exact triangle})
\begin{equation}
\label{eq:blow-up-triangle}
\Hm^\bullet(X)\to\Hm^\bullet(Z)\oplus\Hm^\bullet(\tilde{X})\to\Hm^\bullet(\tilde{Z}).
\end{equation}
\begin{proof}["Proof"]
The functoriality of cohomology easily gives the two morphisms in~(\ref{eq:blow-up-triangle}) so that the composite is zero (this involves introducing a sign, as usual).
We will cheat a little bit and assume that cones are functorial so we get a canonical morphism from the cone of the first map to $\Hm^\bullet(\tilde{Z})$ and it suffices to show this map is invertible.
By \bref{loc}, this in turn can be checked after applying each of $i^*$ and $j^*$.
The upshot of this little game is that we may prove $i^*$(\ref{eq:blow-up-triangle}) and $j^*$(\ref{eq:blow-up-triangle}) are exact triangles.

Let us write the two candidate triangles in terms of the six operations:
\begin{align*}
  i^*\unit\to i^*i_*i^*\unit\oplus i^*p_*p^*\unit\to i^*w_*w^*\unit\\
  j^*\unit\to j^*i_*i^*\unit\oplus j^*p_*p^*\unit\to j^*w_*w^*\unit
\end{align*}
By \bref{loc}, $i_*$ is fully faithful, $j^*i_*=0$, and $j^*w_*=j^*i_*p'_*=0$.
Taking this into account the candidate triangles look as follows:
\begin{align*}
  \unit\to \unit\oplus i^*p_*\unit\to p'_*\tilde{i}^*\unit\\
  \unit\to  j^*p_*\unit\to 0
\end{align*}
It is clear that what remains to do is to `commute $p_*$ with $i^*$ and $j^*$', respectively.
This is precisely the content of proper base change (\Cref{rmk:pbc}).
\end{proof}

\begin{rmk}
\label{rmk:pbc}
Let
\begin{equation}
\label{eq:square}
\begin{tikzcd}
V
\ar[r, "h"]
\ar[d, "k"]
&
X
\ar[d, "f"]
\\
W
\ar[r, "g"]
&
Y
\end{tikzcd}
\end{equation}
be a Cartesian square.
Using the unit and counit of the adjunctions between inverse and direct image functors we deduce a canonical \emph{Beck-Chevalley} (or, \emph{push-pull}) transformation:
\begin{equation}
g^*f_*\to k_*k^*g^*f_*\simeq k_*h^*f^*f_*\to k_*h^*\label{eq:BC}
\end{equation}
\begin{description}
\item[\namedlabel{pbc}{(Proper base change)}] If $f$ is proper then (\ref{eq:BC}) is invertible: $g^*f_*\simeq k_*h^*$.
\end{description}
\end{rmk}

\begin{rmk}
\label{rmk:sbc}
Another instance in which the Beck-Chevalley transformation is invertible is:
\begin{description}
\item[\namedlabel{sbc}{(Smooth base change)}] If $g$ is smooth then (\ref{eq:BC}) is invertible: $g^*f_*\simeq k_*h^*$.
\end{description}
\end{rmk}

\begin{exc}
\label{exc:sbc-alternative}
Recall that by relative purity, the inverse image along a smooth morphism admits a left adjoint.
Construct analogously a transformation
\begin{equation}
\label{eq:sbc}
h_\sharp k^*\to f^*g_\sharp
\end{equation}
and show that it is invertible if and only if~(\ref{eq:BC}) is.
(This is a general phenomenon in 2-category theory: The transformations~(\ref{eq:BC}) and~(\ref{eq:sbc}) are ``mates''.)
\end{exc}

\begin{exc}
\label{exc:nisnevich-triangle}
Consider again the diagram~(\ref{eq:abstract-blow-up}).
Assume now instead that $p$ is \'etale, that $i$ is an open immersion, and that the right vertical arrow is an isomorphism on reduced schemes.
The left part of the diagram in this case is called a \emph{distinguished Nisnevich square}.
Prove in a similar way that there is an associated exact triangle~(\ref{eq:blow-up-triangle}).
(We might call this a \emph{Nisnevich-Mayer-Vietoris triangle}.)
\end{exc}

\begin{exc}
Sometimes, proper and smooth base change instead refer to the following isomorphisms:
\begin{enumerate}[(a)]
\item Proper base change: $g^*f_!\simeq k_!h^*$\label{it:pbc}
\item Smooth base change: $g^!f_*\simeq k_*h^!$\label{it:sbc}
\end{enumerate}
At least morally, these are nothing but reformulations of \bref{pbc} and \bref{sbc}, respectively.
More exactly, using properties discussed previously (for example, localization and relative purity):
\begin{enumerate}[(a)]
\item
Assume $f$ factors as an open immersion followed by a proper morphism (for example, $f$ is separated and of finite type).
Construct a zig-zag of push-pull transformations between $g^*f_!$ and $k_!h^*$.
Show that it is an isomorphism if \bref{pbc} holds.
\item
Assume $g$ factors as a closed immersion followed by a smooth morphism (for example, $g$ is quasi-projective).
Construct a zig-zag of push-pull transformations between $g^!f_*$ and $k_*h^!$.
Show that it is an isomorphism if \bref{sbc} holds.
\end{enumerate}
What can you say about the converse statements?
\end{exc}

\section{What?}
\label{sec:what}

What is a six-functor formalism?
As mentioned in the introduction, we will not try to give a definition.
However, our main goal in this section is to describe an axiomatization of a convenient `stand-in'.
It encodes a minimal set of structure and properties a six-functor formalism is commonly expected to enjoy.
And we show how powerful this notion yet is.
For example, most properties discussed in \Cref{sec:enhancing-cohomology-properties} are consequences, and the few remaining ones (related to duality) can still be studied within this framework. 

The results in this section rely on the work of many mathematicians, see \Cref{rmk:attribution}.
\subsection{A convenient framework}
\label{sec:cosy}
From now on we officially restrict to schemes as our `spaces'.
(But see \Cref{sec:except-funct-rigsh}.)

\begin{notn}
\label{conv:schemes}
Throughout we fix:
\begin{itemize}
\item $\base$: a base scheme, assumed Noetherian and finite-dimensional,
\item $\Sch[\base]{}$: $B$-schemes, assumed separated and of finite-type over~$B$.
\end{itemize}
\end{notn}
If $\base$ is clear from the context or doesn't play a role, we will refer to $\base$-schemes as just schemes and write $\Sch{}$.
Note that all schemes considered are Noetherian and finite-dimensional, and all morphisms are separated and of finite type.
This will come in handy although more general setups are certainly possible.

\begin{rmk}
For our framework to be flexible enough it is better to replace triangulated categories by a suitable enhancement.
We saw a hint of this at a very basic level in the proof of the blow-up triangle~(\ref{eq:blow-up-triangle}).
More serious uses of an enhancement will be made throughout \Cref{sec:what,sec:how}.
We will work with \emph{stable \icats} as developed extensively in~\cite{Lurie_higher-algebra}.
Nevertheless, a reader who is not familiar with this theory may replace them by triangulated categories (or another suitable enhancement) and still get the gist of the text.
Most statements would still make sense and might even be true.
\end{rmk}

\begin{notn}
The \icat of stable \icats and exact functors is denoted by $\iCatst$.
This has a symmetric monoidal structure and we identify commutative algebra objects therein, with symmetric monoidal stable \icats, and we write $\miCatst$ for the \icat of these.
(Note that by our convention, the tensor product is exact in both variables.)
They are an enhancement of tensor-triangulated categories, where our coefficients lived in \Cref{sec:why}.
\end{notn}

Here is the main definition.
While a coefficient lives on a single $B$-scheme we are now interested in the system of all coefficients on all $B$-schemes.
It seems natural to call this data a \emph{coefficient system}.
This terminology was introduced in~\cite{drew:MHM}.
Others have used different terms, see~\Cref{sec:other-approaches}.

\begin{defn}
\label{defn:cosy}%
A \emph{coefficient system (over $\base$)} is a functor
$C:\Sch[\base]{\text{op}}\to\miCatst$ satisfying the following axioms (where we write $f^{*}=C(f)$ for $f$ a morphism of $\base$-schemes).
\begin{enumerate}[(1)]
\item
\namedlabel{ax:left}{(Left)}
\label{Ax:left}%
For each smooth morphism $p:Y\to X\in\Sch[B]{}$, the functor $p^*:C(X)\to C(Y)$ admits a left adjoint $p_\sharp$, and:
\begin{description}
\item[\namedlabel{ax:smooth-bc}{(Smooth base change)}]
For each cartesian square
\[
\begin{tikzcd}
Y'
\ar[r, "p'" above]
\ar[d, "f'" left]
&
X'
\ar[d, "f" right]
\\
Y
\ar[r, "p" above]
&
X
\end{tikzcd}
\]
in $\Sch[B]{}$, the Beck-Chevalley transformation $p'_\sharp (f')^*\to f^*p_\sharp$ is an equivalence.
\item[\namedlabel{ax:projection}{(Smooth projection formula)}] The canonical transformation
\[
p_\sharp(p^*(-)\otimes -)\to -\otimes p_\sharp(-)
\]
is an equivalence.
\end{description}
\item
\namedlabel{ax:right}{(Right)}
For every $X\in\Sch[B]{}$ and every $f:Y\to X$:
\label{Ax:right}%
\begin{description}
\item[\namedlabel{ax:closed}{(Internal hom)}]
The symmetric monoidal structure on $C(X)$ is closed.
\item[\namedlabel{ax:pushforwards}{(Push-forward)}]
The pull-back functor $f^*$ admits a right adjoint $f_*:C(Y)\to C(X)$.
\end{description}

\item
\namedlabel{ax:loc}{(Localization)}
The \icat $C(\emptyset)= 0$ is trivial.
And for each closed immersion $i:Z\into X$ in $\Sch[B]{}$ with complementary open immersion $j:U\into X$, the square (see \Cref{rmk:cosy-def} below)
\begin{equation}
\label{eq:def-cosy-loc}
\begin{tikzcd}
C(Z)
\ar[r, "{i_*}"]
\ar[d]
&
C(X)
\ar[d, "j^*"]
\\
0
\ar[r]
&
C(U)
\end{tikzcd}
\end{equation}
is Cartesian in $\iCatst$.
\item
\label{Ax:hty-stable}%
For each $X\in\Sch[B]{}$, if $p:\affine^1_X\to X$ denotes the canonical projection with zero section $s:X\to\affine^1_X$, then:
\begin{description}
\item[\namedlabel{ax:hty}{($\affine^1$-homotopy)}]
The functor $p^*:C(X)\to C(\affine^1_X)$ is fully faithful.
\item[\namedlabel{ax:stable}{(Tate stability)}]
The composite $p_{\sharp} s_*:C(X)\to C(X)$ is an equivalence.
\end{description}
\end{enumerate}
\end{defn}

\begin{rmk}
\label{rmk:cosy-def}
Let us comment on these axioms and relate them to what we've seen in \Cref{sec:why}.
\begin{enumerate}
\item The existence of left adjoints $f_\sharp$ to inverse images along smooth morphisms is a consequence of relative purity, and we also discussed \bref{ax:projection} in this context.
The \bref{ax:smooth-bc} is another fundamental property although typically formulated as base change of inverse images (along smooth morphisms) against \emph{direct images}.
It was shown in \Cref{exc:sbc-alternative} that these two formulations are equivalent.
\item The structure of a coefficient system encodes only inverse images and tensor products.
The axiom \bref{ax:right} ensures that direct images and internal homs exist as well.
\item Applying \bref{ax:smooth-bc} to the Cartesian square (for $i,j$ as in \bref{ax:loc})
\[
  \begin{tikzcd}
    \emptyset
    \ar[r, "j'"]
    \ar[d, "i'" swap]
    &
    Z
    \ar[d, "i"]
    \\
    U
    \ar[r, "j"]
    &
    X
  \end{tikzcd}
\]
we obtain an equivalence $j'_\sharp (i')^*\isoto i^*j_\sharp$ and the former composite is null-homotopic since $C(\emptyset)=0$.
Taking right adjoints provides the homotopy $j^*i_*\simeq 0$ that is used in~(\ref{eq:def-cosy-loc}).\footnote{I'm grateful to Ryomei Iwasa for pointing out that an explanation of the axiom was required.}
In the presence of \bref{ax:pushforwards}, the square~(\ref{eq:def-cosy-loc}) is Cartesian if and only if the sequence of underlying triangulated categories
\[
\ho(C(Z))\xto{i_*}\ho(C(X))\xto{j^*}\ho(C(U))
\]
is a localization sequence so we recover the condition discussed in \Cref{sec:loc}.
\item The functor $p^*$ in \bref{ax:hty} is fully faithful if and only if the counit $p_\sharp p^*\to\id$ is an equivalence.
As observed in \Cref{exa:motive-smooth}, $p_\sharp p^*=p_!p^!$, and
\[
p_\sharp p^*F=p_\sharp(\unit\otimes p^* F)\isoto p_\sharp\unit\otimes F = p_\sharp p^*\unit\otimes F
\]
by \bref{ax:projection}.
From which one deduces that \bref{ax:hty} is equivalent to the $\aone$-homotopy property considered in \Cref{sec:why} (namely that the homology of the affine line is trivial).
\item Recall that the functor $p_\sharp s_*$ in \bref{ax:stable} was denoted by $\twistns{1}$ in \Cref{sec:why}.
\end{enumerate}
\end{rmk}

\begin{exa}
All theories mentioned in \Cref{sec:why} `should' be examples of coefficient systems.
This has been established for some of them, partially for others.
The only exception to that statement is the bounded derived category of coherent sheaves as usually conceived.
It is not invariant with respect to the affine line, and it does not admit $\sharp$-functoriality.
Nevertheless, there is work in this direction too, see~\cite{scholze:lectures-condensed}.
\end{exa}

In fact, all these examples fall into two important special cases:
\begin{notn}
A coefficient system $C$ is \emph{small} \resp{\emph{presentable}} if the functor takes values in (symmetric monoidal) small \resp{presentable} \icats \resp{and (symmetric monoidal) left-adjoint functors}.
\end{notn}

\begin{rmk}
Stable presentable \icats can be thought of as the \icategorical version of well-generated triangulated categories.
They satisfy a convenient adjoint functor theorem: A functor between presentable \icats is a left adjoint if and only if it preserves colimits.
The \icat of presentable \icats and left adjoint functors is denoted by $\PrL$.
It is anti-equivalent to the \icat of presentable \icats and \emph{right} adjoint functors, denoted~$\PrR$.

It follows that a functor $\Sch[B]{op}\to \mPrst$ automatically satisfies \bref{ax:right}.\footnote{By convention, symmetric monoidal presentable \icats are \emph{presentably symmetric monoidal}, that is, the tensor product commutes with colimits in each variable separately. (A better way of saying this is $\mPrL=\calg{\PrL}$.)}

\end{rmk}

\begin{defn}
\label{defn:cosy-morphism}
\begin{enumerate}
\item
Let $C,D:\Sch[B]{op}\to\miCatst$ be two coefficient systems.
A natural transformation $\phi:C\to D$ is a \emph{morphism of coefficient systems} if, for each smooth morphism $f$ of $B$-schemes, the induced transformation $f_\sharp \phi\to\phi f_\sharp$ is an equivalence.
\item We define the \icat of coefficient systems (over~$B$) as \subicat of the functor category:
\[
\CoSy{B}\subseteq\fun{\Sch[B]{op}}{\miCatst}
\]
One has obvious variants for small and presentable coefficient systems, denoted $\cosy{B}$ and $\CoSyPr{B}$, respectively.
\end{enumerate}
\end{defn}

\begin{exc}
\label{exc:localization-criteria}%
Let $C:\Sch[B]{op}\to\iCatst$ satisfying \bref{ax:smooth-bc} and \bref{ax:pushforwards}.
Show that the following are equivalent:
\begin{enumerate}[(i)]
\item $C$ satisfies \bref{ax:loc}.
\item $C$ satisfies the following three conditions:
\begin{enumerate}[(1)]
\item $C(\emptyset)=0$,\label{it:loc-0}
\item for each closed immersion~$i$, the functor~$i_*$ is fully faithful,\label{it:loc-ff}
\item if $j$ denotes the open immersion complementary to a closed immersion~$i$ then the pair $(i^*,j^*)$ is conservative.\label{it:loc-conservative}
\end{enumerate}
\end{enumerate}
\end{exc}
\subsection{Main result}
\label{sec:cosy-main-result}

The main result comes in two parts: The first wants to say that
\begin{quote}
\emph{coefficient systems underlie six-functor formalisms,}
\end{quote}
and the second wants to say that
\begin{quote}
\emph{morphisms of coefficient systems underlie morphisms of six-functor formalisms.}
\end{quote}
We refer to \Cref{rmk:justification} for the fine print.

\begin{rmk}
\label{rmk:presentable-warning}
If $T$ is a small stable \icat then there is an associated presentable stable \icat $\Ind(T)$, its \emph{Ind-completion}.
As we will discuss in more detail below (\Cref{sta:ind-cosy}), this process turns a small coefficient system into a presentable one.
So while the main results in this section are stated in the presentable context, they are equally true for small, and therefore for `all', coefficient systems (see \Cref{sta:thm-small}).
\end{rmk}

\begin{thrm}
\label{thm:I}
Let $C$ be a presentable coefficient system over~$B$.
Then there are functors (which are equal on objects)
\begin{align*}
  C=C^*:\Sch[B]{\textup{op}}&\to\PrLst\\
  C_!:\Sch[B]{}&\to\PrLst
                 \intertext{with global right adjoints}
                 C_*:\Sch[B]{}&\to\PrRst\\
  C^!:\Sch[B]{\textup{op}}&\to\PrRst
\end{align*}
and for each morphism $f$ of $B$-schemes, a transformation $f_!:=C_!(f)\to C_*(f)=:f_*$  which is invertible when $f$ is proper, satisfying the projection formula, smooth and proper base change, relative purity and `the rest'\ldots
\end{thrm}

\begin{rmk}
We refer to~\cite[Scholie~1.4.2]{ayoub07-thesis-1} or \cite[Theorem~2.4.50]{cisinski-deglise:DM} for more extensive (but still incomplete) lists of properties.
Notably not included in this list is everything on duality, for which see \Cref{rmk:justification} and \Cref{sec:constructibility}.
Some aspects of the proof of \Cref{thm:I} will be discussed in \Cref{sec:how}.
\end{rmk}

\begin{thrm}
\label{thm:II}
Let $\phi:C\to D$ be a morphism of presentable coefficient systems.
Then there are natural transformations
\begin{align*}
  f^*\phi&\isoto\phi f^*\\
   \phi(-)\otimes\phi(-) &\isoto\phi ((-)\otimes (-))\\
   f_!\phi&\isoto\phi f_!\\
  \phi f_*&\to f_*\phi\\
   \phi \ihom(-,-)&\to \ihom(\phi(-), \phi(-))\\
   \phi f^!&\to f^!\phi
\end{align*}
the first three of which are always equivalences, and the last three of which are so `in good cases'.
\end{thrm}

\begin{rmk}
For example, $\phi$ commutes with direct image along \emph{proper} morphisms, and with exceptional inverse image along \emph{smooth} morphisms.
Further `good cases' will be discussed in \Cref{sec:constructibility}.
\end{rmk}

\begin{rmk}
\label{rmk:justification}
\Cref{thm:I,thm:II} are our main justification for viewing coefficient systems as a stand-in for six-functor formalisms.
Let us repeat the caveats already alluded to:
\begin{enumerate}[(a)]
\item The results are on the face of it about presentable coefficient systems only.
But analogous statements can be deduced for small coefficient systems~(\Cref{sta:thm-small}).
And all known examples are either small or presentable.

\item \Cref{thm:I} does not say anything about duality, an important topic in the context of six-functor formalisms (as briefly discussed in \Cref{sec:why}).
This is a consequence of our goal to be as encompassing as possible.
For general coefficient systems, duality cannot be expected unless one restricts to coefficients that are `small' in a certain sense.
This will be taken up again in our short discussion of constructibility (\Cref{sec:constructibility}).
\item The last caveat is related.
Namely, \Cref{thm:II} does not quite say that a morphism of coefficient systems `commutes' with the six operations.
However, in good cases it does so when restricted to `constructible coefficients', see \Cref{sec:constructibility}.
\end{enumerate}
\end{rmk}

\begin{rmk}
\label{rmk:attribution}%
It is clear that in \Cref{thm:I} the extension of a coefficient system $C=C^*$ to $C_!$ is essentially unique (see for details \Cref{sec:!-cs}).
The importance of the axioms \bref{ax:loc} and \bref{ax:hty} in \emph{constructing} the exceptional functoriality was first observed by Voevodsky and was formalized in his notion of \emph{cross-functors}~\cite{voevodsky:cross-functors}.
A version of \Cref{thm:I} was first proved by Ayoub~\cite{ayoub07-thesis-1}.
He worked at the level of triangulated categories and restricted to quasi-projective morphisms.
The latter restriction was removed by Cisinski-D\'eglise~\cite{cisinski-deglise:DM} albeit with an additional axiom.
The homotopy-theoretic difficulties in lifting these results to \icats were addressed by another host of mathematicians, including Liu-Zheng~\cite{liu-zheng:glueing,liu-zheng:6ff-artin-stacks}, Robalo~\cite{robalo:thesis} and Khan~\cite{khan:thesis}.
\end{rmk}

Many consequences can be deduced from \Cref{thm:I,thm:II}.
We refer to \cite{ayoub07-thesis-1,ayoub07-thesis-2,cisinski-deglise:DM} for comprehensive treatments.
As an example, we mention the following result.
It can be viewed as a distillation of the properties discussed in \Cref{sec:blowup} and in the language of \icats it becomes arguably even more powerful.
\begin{cor}
\label{sta:cosy-cdh}
Let $C$ be a presentable coefficient system.
The underlying functor
\[
C:\Sch[B]{\textup{op}}\to\iCat
\]
is a cdh-sheaf.
\end{cor}

\begin{rmk}
For the precise meaning of this statement in general we refer to~\cite[\S\,2]{drew:MHM} or~\cite[\S\,2.3]{AGV}.
In the particular case of the cdh-topology, there is a very convenient criterion, however, for which see the proof below.

In practice, this means that $C$ can be studied locally for the cdh-topology.
For example, if $B=\Spec(k)$ with $k$ a characteristic zero field, then $C$ is uniquely determined by its restriction to $\mathrm{Sm}_k$, the category of smooth $k$-varieties.
\end{rmk}

\begin{proof}
By \bref{ax:loc}, the \icat $C(\emptyset)$ is final.
It then remains to check that $C$ takes distinguished Nisnevich \resp{abstract blow-up} squares to Cartesian squares in $\iCat$.
This amounts essentially to the existence of Nisnevich Mayer-Vietoris triangles and (abstract) blow-up triangles, which we deduced in \Cref{sec:blowup} from Localization and smooth and proper base change.
All of these hold in $C$, by \Cref{thm:I}.

For similar proofs see~\cite[\S\,3.3.a--b]{cisinski-deglise:DM} and~\cite[\S\,6.3]{hoyois:equiv6op}.
\end{proof}

\subsection{Other approaches}
\label{sec:other-approaches}

The framework of coefficient systems is closely related to others in the literature.
Let us summarize some of these relations, without trying to be exhaustive.
\begin{rmk}
\begin{enumerate}
\item A functor $C:\Sch[B]{\textup{op}}\to\miCatst$ is a coefficient system if and only if passing to homotopy categories produces a \textbf{closed symmetric monoidal stable homotopy 2-functor} $\ho(C):\Sch[B]{\textup{op}}\to\mTri$ in the sense of~\cite{ayoub07-thesis-1}.
Similarly, a natural transformation $\phi:C\to D$ between coefficient systems is a morphism of coefficient systems if and only if passing to homotopy categories produces a morphism of symmetric monoidal stable homotopy 2-functors.
\item
A functor $C:\Sch[B]{\textup{op}}\to\miCatst$ is a coefficient system if and only if passing to homotopy categories produces a \textbf{motivic triangulated category} $\ho(C):\Sch[B]{\textup{op}}\to\mTri$ in the sense of~\cite{cisinski-deglise:DM}.
This is not completely obvious since in \loccit an additional axiom (Adj) is assumed.
It follows from \Cref{thm:I} that this axiom is automatic.
\item
It follows from this last observation that presentable coefficient systems also have been considered before, under the name of \textbf{motivic categories of coefficients}~\cite{khan:thesis}.
\item Ultimately, a more complete and thus satisfying framework for six-functor formalisms might be provided by the technology of~\cite{gaitsgory-rozenblyum:dag}, using the \textbf{$(\infty,2)$-category of correspondences}.
However, some of this technology rests on assumptions that are---as far as we are aware---not yet proven in the literature.\footnote{In any case, presentable coefficient systems should extend uniquely to this framework, see~\cite[\S\,4.2]{khan:thesis}.}
\end{enumerate}
\end{rmk}
\subsection{Internal structure of framework}
\label{sec:cosy-structure}
From a bird's-eye view, the framework of coefficient systems consists of cohomology theories and their manifold relations.
For example, Grothendieck's comparison isomorphism between algebraic de\,Rham and Betti cohomology should be reflected in an isomorphism of coefficient systems over $\Spec(\CC)$,
\[
\Dbc\otimes\CC\simeq \Dbrh
\]
that is, by an enhanced version of the Riemann-Hilbert correspondence between (derived) constructible sheaves and regular holonomic $\cD$-modules (cf.\ \Cref{rmk:regulator}).
In particular, extending scalars at the level of cohomology groups is thus reflected by an operation at the level of coefficient systems.

This and many more phenomena should, in other words, be reflected in a rich internal structure of the \icat $\CoSy{}$.
We will be able to provide just a glimpse of this structure if only because mathematicians have barely started to investigate it systematically.

\subsubsection{Initial object}
\label{sec:initial-object}
Let's say we wanted to construct the `universal' coefficient system, that is, the initial object of $\CoSy{}$.
We would probably start with the initial required structure and then would try to freely enforce the axioms of coefficient systems one by one.
As we will see, this can in fact be done, more or less, and the resulting coefficient system turns out to be~$\SH$, (stable) motivic homotopy theory!

\begin{rmk}
One might find this result remarkable.
Without mentioning~$\SH$ in the definition, the \icat of coefficient systems knows about it in a strong sense.
It is probably less remarkable once one remembers that the approach to six-functor formalisms axiomatized in the notion of coefficient systems goes back to Voevodsky's study of the functoriality of $\SH(X)$ in~$X$.
\end{rmk}

We now put this into practice, trying to construct the universal coefficient system.
For more details and generalizations see~\cite{drew-gallauer:usf}.
\begin{cns}
The construction proceeds in several steps.
\begin{enumerate}[(1)]
\item Coefficient systems encode the $()^*$- and $\otimes$-structure.
The initial (as well as final) functor doing so is
\[
\pt:\Sch[]{\textup{op}}\to\miCat
\]
that sends every scheme to the final category with the only possible symmetric monoidal structure.\footnote{Note that stability is a \emph{condition} in the \icategorical world so we will not restrict to stable \icats initially but rather enforce it eventually. (In fact, it will come for free.)}
\item This is not the initial coefficient system because morphisms of coefficient systems are required to commute with $\sharp$-push-forwards.
For example, given a coefficient system $C$ and smooth morphism $p:P\to X$, we have an object
\[
\Hm_\bullet(P):=p_\sharp\unit\in C(X)
\]
and by \bref{ax:smooth-bc} and \bref{ax:projection} we have canonical equivalences
\[
\Hm_\bullet(P)\otimes \Hm_\bullet(P')\cong \Hm_\bullet(P\times_X P')
\]
in $C(X)$.
With some work one can show that
\[
\Hm_\bullet:\mathrm{Sm}_X\to C(X)
\]
defines a functor, symmetric monoidal with respect to the Cartesian structure on smooth $X$-schemes, and this suggests that the functor
\[
X\mapsto (\mathrm{Sm}_X)^\times\in\miCat
\]
is worth a closer look.
In fact, with more work one can show that it is the initial functor satisfying \bref{ax:left}.
\item Passing to the next axiom we see that \bref{ax:right} is not satisfied by this functor.
The only way we know of producing right adjoints in this context is to pass to presentable \icats in order to invoke adjoint functor theorems.
Thus:
\[
X\mapsto\psh{\mathrm{Sm}_X},
\]
the category of presheaves on $\mathrm{Sm}_X$ with the pointwise symmetric monoidal structure (which is the Day convolution in this case).
This forces us to work in the context of presentable \icats from now on though.
(Or at least \icats admitting small colimits.)
\item
It is unclear how one would go about freely enforcing~\bref{ax:loc}.
On the other hand, the axiom seems to be saying that many questions about a coefficient system can be studied locally for the Zariski topology.
And indeed, we saw that it plays an integral role in proving cdh-descent (\Cref{sta:cosy-cdh}).
This suggests that we could get some way towards the axiom by restricting to sheaves for the cdh-topology.
And this `works' except for the fact that the cdh topology isn't a very natural topology on \emph{smooth} schemes.\footnote{See~\cite{khan:cdh} for how to circumvent this problem.}
It turns out that the Nisnevich topology, lying between the Zariski and the cdh-topology, works even better and eventually gives the `same' result:
\[
X\mapsto \dL_{\textup{Nis}}\psh{\mathrm{Sm}_X}.
\]
Here we write $\dL_{\textup{Nis}}$ for the (accessible) localization of presentable \icats with respect to Nisnevich \v{C}ech covers.
\item
Enforcing the next two axioms, \bref{ax:hty} and \bref{ax:stable}, seems comparatively straightforward: We formally invert the canonical projection $\aone_P\to P$ for each smooth $P\to X$, and we formally $\otimes$-invert the cofiber of $P\xto{\infty}\pone_P$.
To make sense of the cofiber it is necessary to pass freely to pointed \icats.\footnote{There are also subtle technical difficulties related to the $\otimes$-structure for which we refer to \cite{robalo:thesis}.}
This doesn't bother us in the least since in the end we want to end up in stable \icats anyway.
Thus we set
\[
X\mapsto 
\left(
\dL_{\aone\cup\textup{Nis}}\psh{\mathrm{Sm}_X}_{\bullet}
\right)[(\pone,\infty)^{\otimes -1}].
\]
The \icat on the right is nothing but $\SH(X)$, the \emph{stable motivic (or $\aone$-)homotopy category} on $X$.
It is a presentable symmetric monoidal \icat.
Note that since $(\pone,\infty)=(\mathbb{G}_{\textup{m}},1)\otimes S^1$ (see \Cref{exc:tate}), the \icat is automatically stable.
\end{enumerate}
\end{cns}

\begin{rmk}
It follows that the resulting functor $\SH:\Sch[]{\textup{op}}\to\mPrst$ has the required shape and it remains to verify the axioms of a coefficient system.
All of them are formal except for \bref{ax:loc}.
The latter is proved by Morel-Voevodsky~\cite{morel-voevodsky-a1} under the name of \emph{Glueing Theorem}.
\end{rmk}

\begin{rmk}
Summarizing, there are at least three ways of thinking about stable motivic homotopy theory:
\begin{enumerate}[(a)]
\item Explicitly, the \icat $\SH(X)$ is constructed as
\[
\left(
\dL_{\aone\cup\textup{Nis}}\psh{\mathrm{Sm}_X}_{\bullet}
\right)[(\pone,\infty)^{\otimes -1}].
\]
\item Robalo~\cite{robalo:thesis} shows that this symmetric monoidal presentable \icat also admits a characterization:
Any $\otimes$-functor $\mathrm{Sm}_X\to D$ into a stable presentable \icat factors uniquely through $\SH(X)$ as soon as it satisfies Nisnevich-excision, $\aone$-invariance, and $\otimes$-inverts $(\pone,\infty)$.
\item By the discussion above~\cite{drew-gallauer:usf}, the coefficient system $\SH$ is the initial object of $\CoSyPr{}$.
\end{enumerate}
This ties back to \Cref{sec:hierarchy-invariants}:
While the second point concerns cohomology theories (at the set-level), the third point is the exact analogue at the  category-level and concerns six-functor formalisms.
\end{rmk}

\subsubsection{Ind-completion}
Since \Cref{thm:I,thm:II} apply to presentable coefficient systems only but many of the coefficient systems considered in \Cref{sec:why} are small, it is very useful to have a process that takes a small coefficient system and outputs a presentable one.
This process is simply Ind-completion (see \Cref{rmk:presentable-warning}).

\begin{prop}
\label{sta:ind-cosy}
There is a functor
\[
\Ind:\cosy{B}\to\CoSyPr{B}
\]
which takes a small coefficient system~$C$ to the functor $X\mapsto \Ind(C(X))$, the target being endowed with the Day convolution product.
\qed
\end{prop}
\begin{exc}
Prove this result.
(A useful fact is that if $f\dashv g$ is an adjunction between small stable \icats, then their unique colimit-preserving extensions $\Ind(f)\dashv \Ind(g)$ again form an adjunction between their Ind-completions.)
\end{exc}

\begin{rmk}
\label{rmk:small-ind}
So, given a small coefficient system~$C$, we apply \Cref{thm:I} to $\Ind(C)$ and obtain the full six operations on the system of \icats $\Ind(C(X))$.
However, at this point we do not know whether the exceptional functoriality restricts to the subsystem $C(X)\subset \Ind(C(X))$.
It turns out that it does and the proof is not difficult.
\end{rmk}

\begin{lem}
\label{sta:!-preserve}
Let $f:X\to Y$ be a morphism of $B$-schemes, let $M\in C(X)$ and $N\in C(Y)$.
Then
\begin{align*}
f_!M\in C(Y),&& f^!N\in C(X).
\end{align*}
\end{lem}
\begin{proof}[Proof sketch]
For the first statement we factor $f$ as an open immersion followed by a proper morphism and reduce to proving each case separately.
In the latter case we have $f_!M=f_*M\in C(Y)$ and we win.
In the former we have $f_!M=f_\sharp M\in C(Y)$ and we win again.

For the second statement we use \bref{ax:loc} to show that an object $L\in \Ind(C(X))$ belongs to $C(X)$ if (and only if) $L|_{U_i}\in C(U_i)$ for some open cover $(U_i)$ of~$X$.
In other words, the question is local on~$X$.
In particular, we can assume that $f$ is quasi-projective, and factor it as a closed immersion followed by a smooth morphism.
The first case then follows from \bref{ax:loc} and the second case follows from relative purity.
\end{proof}

\begin{exc}
Fill in the details of this proof sketch.
\end{exc}

\begin{cor}
\label{sta:thm-small}
\Cref{thm:I,thm:II} admit analogues for small coefficient systems.\footnote{Of course, in this case the four functors $()^*, ()_*, ()_!, ()^!$ take values in \emph{small} stable \icats.}
\end{cor}
\subsubsection{Constructibility}
\label{sec:constructibility}
Let $C$ be a coefficient system.

\begin{notn}
Denote by $\Cgm(X)\subset C(X)$ the smallest full \subicat that
\begin{enumerate}[(1)]
\item contains $f_\sharp\one\twist{n}$ for $f:Y\to X$  smooth, $n\in\Z$, and
\item is stable and closed under direct factors (we call such subcategories \emph{thick}).
\end{enumerate}
This defines the subfunctor $\Cgm\subseteq C$ \emph{of geometric origin} (see \Cref{sta:gm-sub}).
\end{notn}

\begin{lem}
\label{sta:gm-sub}
$\Cgm\subseteq C:\Sch[]{\textup{op}}\to\miCatst$ is a subfunctor and the inclusion $\Cgm\to C$ commutes with $f_\sharp$ for $f$ smooth.
\end{lem}
\begin{proof}
This follows immediately from \bref{ax:smooth-bc} and \bref{ax:projection}.
\end{proof}

\begin{exa}
\label{exa:beilinson-motives}
\begin{enumerate}
\item For $C=\SH$, the geometric part coincides with the compact part: The objects of $\SH^{\textup{gm}}(X)$ are precisely the compact objects in $\SH(X)$.\footnote{Recall that in an \icat $\cC$ with filtered colimits, an object $M$ is compact if $\map[\cC]{M}{-}:\cC\to\mathcal{S}$ to the \icat of spaces preserves filtered colimits. If $\cC$ is a stable \icat, this can be tested at the level of homotopy categories and is equivalent to $M$ being compact in the sense of triangulated categories: $\hom_{\ho(\cC)}(M,-)$ preserves direct sums.}
Moreover, $\SH$ is compactly generated so that $\Ind(\SH^{\textup{gm}})=\SH$.
\item The same is true for $C=\DMB$, Beilinson motives in the sense of~\cite{cisinski-deglise:DM} (or rather, their \icategorical enhancement).
That is, $\DMBgm$ is the compact part and $\Ind(\DMBgm)=\DMB$.
Moreover, if $k$ is a field there is a canonical equivalence
\[
\DMBgm(\Spec(k))=\DM^{\textup{gm}}(k;\Q)
\]
with (the \icategorical enhancement of) Voevodsky's category of geometric motives with rational coefficients~\cite{voevodsky00-mm}.
\item In the $\ell$-adic setting, `of geometric origin' is close to `bounded-constructible', see~\cite{cisinski-deglise:etale-motives}.
\end{enumerate}
\end{exa}

\begin{rmk}
Let $C$ be a presentable (or just cocomplete) coefficient system.
By \Cref{sec:initial-object}, there is a unique morphism of coefficient systems $\SH\to C$, and $\Cgm\subseteq C$ is exactly the thick subfunctor generated by the image of $\SH^{\textup{gm}}$.
\end{rmk}

We should now address the question whether $\Cgm$ is a coefficient system as well.
We will not state sufficient conditions here and refer to the literature instead:

\begin{thrm}[{\cite[\S\,3]{ayoub:betti}, \cite[\S\,4.2]{cisinski-deglise:DM}}]
In `good cases', $\Cgm$ is a coefficient system.
\end{thrm}

\begin{exa}
An example to which the theorem applies is Beilinson motives (\Cref{exa:beilinson-motives}).
In particular one obtains, in this case, a very satisfying picture translating between small and compactly generated coefficient systems:
\[
\begin{tikzcd}
\cosy{}\ni\DMBgm
\ar[r, "{\Ind}", bend left=30, mapsto]
&
\DMB\in\CoSyPr{}
\ar[l, "{(-)^{\textup{gm}}}", bend left=30, mapsto]
\end{tikzcd}
\]
\end{exa}

\begin{cor}
In the same `good cases' assume $\phi:C\to D$ is a morphism of coefficient systems.
Then
\[
\phi|_{\Cgm}:\Cgm\to D
\]
commutes with all six functors.
\end{cor}

\begin{rmk}
This improves on \Cref{thm:II} in `good cases'.
\end{rmk}

\begin{rmk}
There is a more general notion of \emph{constructibility} for coefficient systems over~$B$.
Instead of the generating set $\{f_\sharp\one\twist{n}\}$ one may consider the set $\{f_\sharp p^*F\}$ where $F$ runs through a specified set of coefficients on the base~$B$, $f:Y\to X$ is smooth, and $p:Y\to B$ is the structure morphism.
We recover the geometric part by allowing only Tate twists as coefficients on~$B$.

The more general notion is useful in the study of duality phenomena, one of the topics of \Cref{sec:why} which wasn't addressed by the notion of coefficient systems alone.
We refer again to~\cite[\S\,3]{ayoub:betti} and~\cite[\S\,4.2]{cisinski-deglise:DM} for in-depth discussions.
\end{rmk}

\subsubsection{Miscellanea}
\label{sec:miscellanea}

Many other topics could be discussed in the framework of coefficient systems, for example:
\begin{enumerate}[(a)]
\item
We saw in \Cref{sta:cosy-cdh} that coefficient systems satisfy cdh-descent.
Some of them satisfy \textbf{descent} with respect to stronger topologies, however, such as \'etale descent (and therefore eh-descent) or h-descent~\cite[\S\,3]{cisinski-deglise:DM}.
This can be useful in extending coefficient systems from schemes to algebraic stacks via an atlas, say.
\item It makes sense to consider \textbf{linear} coefficient systems and scalar extension.
For example, in some cases being $\Q$-linear implies h-descent~\cite[\S\,3.3.d]{cisinski-deglise:DM}.
A general discussion of scalar extension can be found in~\cite[\S\,8]{drew:MHM}, and we will discuss one application of this technique in \Cref{sec:motiv-cosy}.
\item \textbf{Orientable} coefficient systems are somewhat simpler to work with in the sense that `all Thom twists are Tate twist' (see \Cref{exa:thom-ladic} and~\cite[\S\,2.4.c]{cisinski-deglise:DM}).
\end{enumerate}
In these and many other cases there should be corresponding initial objects (similarly to \Cref{sec:initial-object}).

Let us mention just two instances of possibly more surprising phenomena.
As remarked at the beginning of this section, clearly, a lot remains to be explored!
\begin{exa}
There is a functor
\[
\exp:\CoSy{B}\to\CoSy{B}
\]
which `exponentiates' a coefficient system, and that we study in forthcoming work~\cite{FGPL:exp-cs}.
When applied to $\DMB$ it produces a new coefficient system $\DM_{\mathcyr{B}}^{\textup{exp}}$ that should enhance Fres\'an-Jossen's theory of exponential motives~\cite{fresan-jossen:exponential-motives}.\footnote{More precisely, $\DM_{\mathcyr{B}}^{\textup{exp}}(k)$ bears to their theory the same relation as $\DMB(k)$ to Nori motives, for $k\subseteq\mathbb{C}$ a field.}
And when applied to mixed Hodge modules, it should produce an enhancement of Kontsevich-Soibelman's exponential mixed Hodge structures~\cite{Kontsevich-Soibelman:exponential-HS}.
An interesting aspect of this construction is that every exponentiated coefficient system comes with an additional `seventh' operation, the Fourier transform familiar from the $\ell$-adic theory as well as $\cD$-modules (see e.g.~\cite{MR823177,MR1081536}).

\end{exa}

\begin{exa}
\label{exa:weil-sheaves-construction}
Let $\FFq$ be a finite field and choose an algebraic closure~$\FF$.
If $C$ is a coefficient system on $\FF$-schemes, one can define a functor
\[
C^{\textup{W}}:\Sch[\FFq]{op}\to\miCatst
\]
by the formula, for any $\FFq$-scheme $X$,
\[
C^{\textup{W}}(X)=\lim 
\left(
  \begin{tikzcd}
    C(X\times_{\FFq}\FF)
    \ar[r, shift left=.3em, "{\textup{Fr}}"]
    \ar[r, shift right=.3em, "{\textup{id}}" swap]
    &
    C(X\times_{\FFq}\FF)
\end{tikzcd}
\right)
\]
where $\textup{Fr}$ denotes the $q$-Frobenius on~$X$ and the limit is taken in $\miCatst$.
The superscript is in honour of Weil since in the case of $\ell$-adic cohomology, the \icat $C^{\textup{W}}(X)$ can be seen as a derived category of Weil sheaves~\cite{hemo-richarz-scholbach:kunneth}.
With some work (see \Cref{exc:weil-sheaves} below) one shows that this underlies a functor
\[
(-)^{\textup{W}}:\CoSy{\Spec(\FF)}\to\CoSy{\Spec(\FFq)}.\footnote{This example was brought to my attention by Joshua Lieber.}
\]
\end{exa}

\begin{exc}
\label{exc:weil-sheaves}
The goal of this extended exercise is to prove that $C^{\textup{W}}$ of \Cref{exa:weil-sheaves-construction} is a coefficient system.
This can be done with the following steps:
\begin{enumerate}[(1)]
\item Let $\omega:C\to D$ be a natural transformation of functors $\Sch[B]{op}\to\miCatst$ and assume that:
\begin{enumerate}[({1}.1)]
\item $D$ is a coefficient system,
\item $C$ admits left adjoints $p_\sharp$ for smooth morphisms $p$, and $\omega$ commutes with them,
\item $\omega_X:C(X)\to D(X)$ is conservative for each $X\in\Sch[B]{op}$.
\end{enumerate}
Show that $C$ satisfies \bref{ax:hty}, \bref{ax:smooth-bc}, \bref{ax:projection} as well as \Cref{it:loc-0,it:loc-conservative} of \Cref{exc:localization-criteria}.
\item Assume in addition that $C$ satisfies \bref{ax:right} and that $\omega$ commutes with $f_*$ for all immersions~$f$.
Then $C$ also satisfies \bref{ax:loc} and \bref{ax:stable}, hence is a coefficient system.
\item Use the previous point to show that $C^{\textup{W}}$ of \Cref{exa:weil-sheaves-construction} is a coefficient system. \textit{Hint}: For any diagram $F:I\to\iCatst$ the canonical functor $\lim_I F\to\prod_{i\in I_0}F(i)$ is conservative.
\end{enumerate}
\end{exc}

\section{How?}
\label{sec:how}

The question alluded to in the title can be understood in at least two ways, both of which we want to address partially here:
\begin{enumerate}[(A)]
\item how to construct six-functor formalisms in general?
\label[question]{interpretation-1}
\item how to obtain six-functor formalisms from coefficient systems? that is, how to prove \Cref{thm:I}?
\label[question]{interpretation-2}
\end{enumerate}
The two are related.
Often the $\otimes$-structure and $*$-functoriality is produced without much effort and it is the $!$-functoriality that poses the most serious difficulties.
Below we will focus on the problem of constructing exceptional direct and inverse images, and we will refer to the literature for the problem of proving the expected properties.

\subsection{Exceptional functoriality for coefficient systems}
\label{sec:!-cs}
We start with \Cref{interpretation-2} and for this we want to follow the strategy employed by Deligne to produce exceptional functoriality in $\ell$-adic cohomology~\cite[\S\,XVII.3, 5.1]{MR0354654}.
As we will see, working in the generality we do, additional difficulties arise that need to be addressed.

\begin{rmk}
Let $f$ be a morphism of $B$-schemes.
Since $f$ is separated and of finite type (\Cref{conv:schemes}), we may use Nagata compactification to find a factorization
\begin{equation}
\label{eq:factorization-f}
\begin{tikzcd}
\ar[rr, "f"]
\ar[rd, "j" below]
&&
{}
\\
&
\ar[ru, "p" below]
\end{tikzcd}
\end{equation}
with $j$ an open immersion and $p$ a proper morphism.
We would then like to set
\[
f_!:=p_*j_\sharp
\]
but this definition poses several difficulties:
\begin{enumerate}[(1)]
\item \emph{well-definedness}: is it `independent' of the factorization?
\item \emph{right-adjoint}: why is there a right adjoint~$f^!$?
\item \emph{functoriality}: in what sense is it functorial in~$f$?
\end{enumerate}
\end{rmk}
We will address each of these difficulties in turn.
\subsubsection{Well-definedness}
\label{sec:well-definedness}
Consider the category $\Comp(f)$ of compactifications of~$f$: Its objects are factorizations as in~\eqref{eq:factorization-f}, with morphisms (necessarily proper) making the obvious diagram commute:
\[
\begin{tikzcd}
&\ar[dd, "r"]
\ar[rd, "p_1"]
\\
\ar[ru, "j_1"]
\ar[rd, "j_2" {yshift=-.8ex, below, pos=0.3}]
&&
{}
\\
&
\ar[ru, "p_2" {yshift=-.8ex, below, pos=0.7}]
\end{tikzcd}
\]
The category~$\Comp(f)$ is easily seen to be cofiltered.
In comparing $(p_1)_*(j_1)_\sharp$ with $(p_2)_*(j_2)_\sharp$ we may therefore assume a morphism~$r$ as above.
We then find
\[
(p_1)_*(j_1)_\sharp\simeq (p_2)_*r_*(j_1)_\sharp\stackrel{!}{\simeq}(p_2)_*(j_2)_\sharp
\]
where the last identification would follow if we could `commute' $j_\sharp$'s with $p_*$'s.
This is known as the \emph{support property}:
\begin{description}
\item[\namedlabel{supp}{(Support)}] Given a Cartesian square
\[
\begin{tikzcd}
\ar[r, "j_1"]
\ar[d, "r_2" left]
&
\ar[d, "r_1"]
\\
\ar[r, "j_2" below]
&
{}
\end{tikzcd}
\]
with $r_1$ proper and $j_2$ an open immersion,
the induced transformation $(j_2)_\sharp (r_2)_*\to (r_1)_*(j_1)_\sharp$ is an equivalence.
\end{description}

\begin{exc}
\label{exc:pbc-supp}
Show that \bref{pbc}$\ \Rightarrow$\ \bref{supp}.
\end{exc}

\begin{rmk}
In view of \Cref{exc:pbc-supp} it is natural to try to establish \bref{pbc}.
We now sketch the main ideas that go into deducing it from the axioms of a coefficient sytem.
The same strategy will also be employed in \Cref{sec:right-adjoint}.
\begin{enumerate}[(1)]
\item Recall that we are given~(\ref{eq:square}) with $f$ proper, and would like to show the transformation $g^*f_*\to k_*h^*$ from~(\ref{eq:BC}) to be an equivalence.
By \bref{ax:loc} and Chow's lemma we reduce to $f$ projective, $f=pi$ where $i$ is a closed immersion and $p:\mathbb{P}^d_Y\to Y$ is the canonical projection.
This reduction step is written out in detail in~\cite[4.1.1.(1)]{AGV}.
\item The case of~$i_*$ (`closed base change') follows easily from \bref{ax:loc} so we further reduce to the case of~$p_*$.
\item Hence, in addition to being projective, $p$ is also smooth of relative dimension~$d$ so we expect to observe Atiyah duality~(\Cref{exc:atiyah}).
In other words, one ought to be able to show a canonical equivalence
\begin{equation}
\label{eq:atiyah-duality-Pd}
p_*\simeq p_\sharp\twist{-T_p}.
\end{equation}
Although this is a rather explicit problem, the proof is long and involved.
Moreover, constructing a candidate for the equivalence also involves proving some form of purity.
We refer to \cite[Th\'eor\'eme~1.7.9]{ayoub07-thesis-1} for details. 
\item
Having the equivalence~(\ref{eq:atiyah-duality-Pd}) at our hand, we are reduced to show that both $p_\sharp$ and $\twist{T_p}$ `commute with inverse images'.
This is exactly \bref{ax:smooth-bc} and closed base change (recall \Cref{rmk:thom-construction}).
\end{enumerate}
\end{rmk}
\subsubsection{Right-adjoint}
\label{sec:right-adjoint}
It is clear that the functor $j_\sharp$ admits a right adjoint, namely~$j^*$.
To show that $p_*$ does as well (for $p$ proper) we will use the adjoint functor theorem for presentable \icats.
In other words we will show that $p_*$ preserves colimits.
The advantage of this formulation of the problem is that it becomes amenable to the same attack as the one employed in proving \bref{pbc} above:
One reduces to projective and further to smooth projective morphisms and then obtains the identification~(\ref{eq:atiyah-duality-Pd}).
Both of the functors on the right are left adjoints and we conclude.

\subsubsection{Functoriality}
\label{sec:functoriality}
Well-definedness discussed in \Cref{sec:well-definedness} is only one aspect of the problem that is posed by functoriality.
Recall that we want to construct a functor $C_!:\Sch[B]{}\to\PrLst$.
Deligne was able to achieve this at the level of triangulated categories by setting, for $f:X\to Y$,
\begin{equation}
\label{eq:comp-filtered-colim}
f_!:=\varinjlim_{(p,j)\in\Comp(f)^{\textup{op}}}p_*j_\sharp,
\end{equation}
using that $\Comp(f)$ is cofiltered and the functor $*\circ\sharp:\Comp(f)^{\textup{op}}\to\Hom(C(X),C(Y))$ sends morphisms to isomorphisms, by \bref{supp}.
But even constructing such a functor $*\circ\sharp$ is a daunting task in the context of \icats as it would involve providing, in addition to the homotopies of \bref{supp}, homotopies between these and so on \textsl{ad infinitum}.

\begin{rmk}
One solution to this homotopy theoretic problem was developed in~\cite{liu-zheng:glueing}, based on multisimplicial sets.
It is very general but unfortunately rather complicated.
We would like to describe a more elementary solution specific to the given problem.
It is based on our recent collaboration with Ayoub and Vezzani~\cite{AGV}.
\end{rmk}

\begin{rmk}
\label{rmk:!-cns-idea}
The basic idea is very simple.
Let $f:X\to Y$ be a morphism of $B$-schemes which admits a compactification~$\bar{f}$:
\begin{equation}
\label{eq:compactification-f}
\begin{tikzcd}
X
\ar[d, "j"]
\ar[r, "f"]
&
Y
\ar[d, "k"]
\\
\bar{X}
\ar[r, "\bar{f}"]
&
\bar{Y}
\end{tikzcd}
\end{equation}
Here, the square commutes, $j$ and $k$ are open immersions and $\bar{X}$ and $\bar{Y}$ are proper $B$-schemes.
We obtain a diagram of solid arrows
\[
\begin{tikzcd}
C(X)
\ar[d, "j_\sharp"]
\ar[r, dotted, "f_!"]
&
C(Y)
\ar[d, "k_\sharp"]
\\
C(\bar{X})
\ar[r, "\bar{f}_*"]
&
C(\bar{Y})
\end{tikzcd}
\]
and as $\bar{f}$ is proper we would like to define $f_!$ so that the square `commutes'.
Again, this is not a tenable strategy in the context of \icats.
Being commutative is not a property but a structure and we are back to the exact same issue as before.

However, we can avoid this issue with the following trick.
Define a full \subicat $C(X,\bar{X})_!$ of $C(\bar{X})$ as the essential image of~$j_\sharp$, and similarly for~$k$.
(In particular, we have an equivalence $C(X)\simeq C(X,\bar{X})_!$.)
\bref{supp} implies that $\bar{f}_*$ restricts to a morphism $C(X,\bar{X})_!\to C(Y,\bar{Y})_!$.
The gain is that the functor $\bar{f}_*$ is already part of a functor $C_*:\Sch[B]{}\to\PrR$ which encodes these higher homotopies.

After outlining the basic idea we can now summarize the construction of the functor~$C_!$.
\end{rmk}

\begin{cns}
We will use the following diagram
\[
\begin{tikzcd}
&
\Comp[B]
\ar[ld, "\omega" {swap, yshift=.3ex}]
\ar[rd, "\pi"]
\\
\Sch[B]{}
&&
\Sch[B]{\textup{prop}}
\end{tikzcd}
\]
in which $\Comp[B]$ denotes the category whose objects are pairs $(X,\bar{X})$ as above and whose morphisms are pairs $(f,\bar{f})$ as in~(\ref{eq:compactification-f}).
Forgetting $\bar{X}$ \resp{$X$} defines the functor $\omega$ \resp{$\pi$}.
(Here, $\Sch[B]{\textup{prop}}$ denotes the category of $B$-schemes and proper morphisms.)

Starting with $C$ and passing to $C_*$ as above we obtain the functor $C_*\circ \pi:\Comp[B]\to\PrR$ that informally can be described as follows:
\begin{align*}
  (X,\bar{X})&\longmapsto C(\bar{X})\\
  (f,\bar{f})&\longmapsto \bar{f}_*
\end{align*}
In fact, this functor takes values in $\PrL$ as well, by \Cref{sec:right-adjoint}.

If $C(-,-)_!:\Comp[B]\to\PrL$ denotes the full subfunctor of $C_*\circ \pi$ considered in \Cref{rmk:!-cns-idea} then we define
\[
C_!:=\mathrm{LKE}_\omega C(-,-)_!:\Sch[B]{}\to\PrL,
\]
the left Kan extension along~$\omega$.
This last step thus removes the dependency of the factorization in a similar way as in~(\ref{eq:comp-filtered-colim}).
\end{cns}

\begin{rmk}
It is not difficult to prove that $C_!(f)$ recovers $p_*j_\sharp$ up to homotopy for any factorization as in~(\ref{eq:factorization-f}).
For details, we refer to~\cite[\S\,4.3]{AGV}.
From there, one can go on and prove the expected properties of this six-functor formalism, see~\cite{ayoub07-thesis-1} or~\cite{cisinski-deglise:DM}.
\end{rmk}

\subsection{Motivic coefficient systems}
\label{sec:motiv-cosy}

Once one has \Cref{thm:I} at one's disposal, of course, \cref{interpretation-1} becomes: how to construct coefficient systems? In this brief section we will describe an elegant and powerful procedure that has been employed in the literature to produce `motivic' coefficient systems.
This topic would have just as well fit in with \Cref{sec:cosy-structure}.

\begin{rmk}
\label{rmk:represent-cohomology}
Let $C\in\CoSyPr{B}$ be a coefficient system, say presentable to fix our ideas.
As we saw in \Cref{sec:initial-object}, there is an essentially unique morphism $\rho_C^*:\SH\to C$ from the initial object which can be viewed as the (homological) $C$-realization: For any $B$-scheme~$X$, the functor $\rho^*_C(X):\SH(X)\to C(X)$ sends a smooth $X$-scheme~$Y$ to its homology coefficient $\Hm_\bullet(Y)$ in $C(X)$.
This functor admits a right adjoint $\rho_*^C(X):C(X)\to \SH(X)$ that has a canonical lax symmetric monoidal structure (since $\rho^*_C(X)$ underlies a symmetric monoidal functor).
In particular we see that $\rho_*^C(B)\unit\in\calg{\SH(B)}$ is a motivic ring spectrum which we denote by~$\cC$.

This object represents $C$-cohomology in the sense that for any smooth $B$-scheme~$X$, we have by adjunction
\[
\pi_0\map[\SH(B)]{X}{\cC(m)[n]}\simeq\pi_0\map[C(B)]{\Hm_\bullet(X)}{\unit(m)[n]}\simeq\Hm^{n}(X;\unit(m)).
\]
The observation we want to make now is that \emph{every} motivic ring spectrum represents some cohomology theory.
\end{rmk}

\begin{notn}
\label{cns:motivic-cosy}
Let $\cA\in\calg{\SH(B)}$ be a motivic ring spectrum and denote (abusively) by $\cA_X:=f^*\cA\in\calg{\SH(X)}$ its pull-back to any $B$-scheme $f:X\to B$.
The association $X\mapsto \Mod[\cA_X]{\SH(X)}=:\SH(X;\cA)$ underlies a functor
\begin{equation}
\label{eq:modules-cosy}
\SH(-;\cA):\Sch[B]{op}\to\mPrst
\end{equation}
that - in anticipation of the next \namecref{sta:modules-cosy} - we call the \emph{motivic coefficient system represented by~$\cA$}.
\end{notn}

\begin{thrm}
\label{sta:modules-cosy}
The functor $\SH(-;\cA)$ of~(\ref{eq:modules-cosy}) is a presentable coefficient system and the canonical `free functor' $\rho^*_{\cA}:\SH\to\SH(-;\cA)$ is a morphism of presentable coefficient systems.
\end{thrm}

A proof of this result can be found in~\cite[Theorem~8.10]{drew:MHM}, see also~\cite[\S\,7.2, 17.1]{cisinski-deglise:DM}.

\begin{rmk}
We may now combine the constructions of \Cref{rmk:represent-cohomology} and \Cref{cns:motivic-cosy}.
That is, in the situation of \Cref{rmk:represent-cohomology} we obtain a factorization of $\rho_C^*:\SH\to C$ through the \emph{motivic coefficient system associated with $C$}:
\[
\SH\xto{\rho^*_{\cC}}\SH(-;\cC)\xto{\tilde{\rho}_C^*}C
\]
The induced functor $\tilde{\rho}^*_C$ factors further through the localizing subfunctor $\tilde{C}$ of $C$ generated by the part of geometric origin (\Cref{sec:constructibility}).
By a tilting argument, the resulting morphism
\[
\tilde{\rho}_C^*:\SH(-;\cC)\to\tilde{C}
\]
is in fact an equivalence in some cases of interest, see~\cite[Theorem~17.1.5]{cisinski-deglise:DM}.
\end{rmk}

\begin{rmk}
  In summary, we have procedures
\[
\begin{tikzcd}
\calg{\SH(B)}
\ar[r, "{\SH(-;-)}", bend left=30]
&
\CoSyPr{B}
\ar[l, "{\rho_*(B)\unit}", bend left=30]
\end{tikzcd}
\]
which can be upgraded to functors.
We could ask whether these describe a colocalization, that is, whether they form an adjunction in which the left adjoint~$\SH(-;-)$ is fully faithful.
\end{rmk}

\begin{exa}[{\cite[17.1.7]{cisinski-deglise:DM}}]
\label{exa:betti-modules}
Consider the Betti realization functor
\[
\rho^*_{\textup{B}}:\SH(\Spec(\CC))\to\D(\Q)
\]
that sends a smooth complex scheme $X$ to the rational singular chain complex $\mathrm{Sing}(X^{\textup{an}})\otimes\Q$ on the underlying complex analytic space.
It is naturally symmetric monoidal - in fact, it is part of a morphism of coefficient systems on complex schemes~\cite{ayoub:betti}:
\[
\SH\to\D((-)^{\textup{an}};\Q)
\]
The associated motivic ring spectrum $\cB:=\rho^B_*\Q\in\calg{\SH(\CC)}$ is the \emph{(rational) Betti spectrum} that represents Betti cohomology.
We will now describe the resulting coefficient system $\SH(-;\cB)$ more explicitly.

First, observe that for a general complex scheme~$X$, the functor
\[
\tilde{\rho}_{\textup{B}}^*(X):\SH(X;\cB)\to\D(X^{\textup{an}};\Q)
\]
is far from an equivalence.
Instead, it factors through
\[
\SH(X;\cB)\to\Ind(\Dbc(X;\Q))\to\D(X^{\textup{an}};\Q)
\]
where the second arrow is the colimit-preserving functor extending the identity on $\Dbc(X;\Q)$.
The first functor in this factorization is in fact fully faithful, and the image is generated under colimits, desuspensions and truncations (with respect to the canonical t-structure) by sheaves of the form $f_*\Q$, where $f:Y\to X$ is proper, see~\cite[Theorem~1.93]{ayoub-anabel}. 
\end{exa}

\begin{exa}[{\cite{drew:MHM}}]
\label{exa:hodge-modules}
Saito's derived categories of mixed Hodge modules do not, in an obvious way, admit an enhancement to a coefficient system.
(As a result, the Hodge realization functors are not known to commute with the six functors on compact objects.)
On the other hand, there is a Hodge realization functor
\[
\rho_{\textup{H}}^*:\SH(\Spec(\CC))\to\D(\Ind(\mathrm{MHS}^{\textup{p}}_\Q))
\]
with values in the derived \icat of Ind-completed polarizable mixed Hodge structures over~$\Q$.
The associated motivic ring spectrum $\cH:=\rho^{\textup{H}}_*\Q(0)$ is the \emph{absolute Hodge spectrum} that represents absolute Hodge cohomology.
Drew calls the resulting coefficient system $\SH(-;\cH)$ \emph{motivic Hodge modules}, and they satisfy many of the properties expected of a coefficient system that should capture mixed Hodge modules of geometric origin.
In line with this, he conjectures that for each complex scheme~$X$, the triangulated category of compact objects in $\ho(\SH(X;\cH))$ embeds fully faithfully into Saito's $\Db(\MHM(X))$.
\end{exa}

\begin{exa}[{\cite{MR3865569}}]
As in \Cref{exa:hodge-modules}, until recently there was no known enhancement of Voevodsky's category of motives over a field, $\DMgm(\Spec(k);\Z)$ to a coefficient system in mixed characteristic.
(The situation was better understood with rational coefficients and/or in equal characteristic.)
Spitzweck constructs a motivic ring spectrum $\cM\in\SH(\Spec(\Z))$ that represents Bloch-Levine motivic cohomology and then defines
\[
\SH(-;\cM):\Sch[]{op}\to\mPrst
\]
that can be seen as a coefficient system of integral motivic sheaves.
He proves that over a field~$k$ the compact part of $\SH(\Spec(k);\cM)$ is equivalent to $\DMgm(\Spec(k);\Z)$,
while with rational coefficients and for any scheme~$X$ one recovers Beilinson motives:
\[
\SH(X;\cM\otimes\Q)\simeq\DMB(X)
\]

\end{exa}
\subsection{Exceptional functoriality for \texorpdfstring{$\RigSH$}{RigSH}}
\label{sec:except-funct-rigsh}

We have two goals for this last section.
First, we want to say something regarding \Cref{interpretation-1} at the beginning of \Cref{sec:how}.
And secondly, we want to give an example of a six-functor formalism outside the world of schemes (and topological spaces) that have dominated the discussion so far.
\begin{rmk}
In the context of schemes, \Cref{thm:I} provides a very useful criterion for recognizing six-functor formalisms.
In contexts that are not too different from schemes one can hope to establish a similar criterion, see e.g.~\cite{khan-ravi:sh-stacks} for (certain) algebraic stacks.
However, in general one shouldn't expect the axioms of coefficient systems---even if interpreted appropriately---to be sufficient to guarantee the existence of $!$-functoriality.

Rigid (or `non-archimedean') analytic geometry is arguably an example of a theory that is too different for a successful transfer.
In the following pages we want to describe how a different kind of transfer allows one to construct $!$-functors (and prove the expected properties) on the `universal' rigid-analytic theory, namely rigid-analytic stable motivic homotopy theory~$\RigSH$.
This is a report on the work with Ayoub and Vezzani~\cite{AGV} already mentioned in \Cref{sec:functoriality}.
\end{rmk}

\begin{rmk}
Rigid-analytic geometry is the analogue of complex-analytic geometry over non-archimedean fields, e.g.\ $p$-adic fields.
The theory retains both algebraic and analytic aspects, and it has found many applications in arithmetic algebraic geometry, particularly in the wake of Scholze's work on perfectoid spaces and $p$-adic geometry.

Rigid-analytic spaces in the sense this term is used in~\cite{AGV} form a category $\RigSpc[]{}$ that encompasses both Tate's rigid-analytic varieties (and Berkovich spaces) as well as a large class of adic spaces (e.g.~all `stably uniform' ones~\cite{MR3824781}) in the sense of Huber.
While ridding the treatment of unnecessary Noetherianity assumptions was a goal of~\cite{AGV}, these technical details will not concern us in this short outline.
\end{rmk}

\begin{rmk}
The construction of $\RigSH$ is originally due to Ayoub~\cite{ayoub:rigid-motives} and modeled on Morel-Voevodsky's construction of~$\SH$ (see \Cref{sec:initial-object}).
In fact, the two are entirely parallel according to the following `dictionary':
\[
\begin{tikzcd}[row sep=tiny]
\Sch{}
\ar[r, squiggly, leftrightarrow]
&
\RigSpc{}
\\
\aone
\ar[r, squiggly, leftrightarrow]
&
\mathbb{B}^1
\\
\mathbb{G}_{\textup{m}}
\ar[r, squiggly, leftrightarrow]
&
\mathbb{T}
\end{tikzcd}
\]
Here $\mathbb{B}^1$ is the closed unit ball and $\mathbb{T}\subset\mathbb{B}^1$ the annulus.

Unsurprisingly and in a completely parallel fashion, $\RigSH$ comes with a closed symmetric monoidal structure and $*$-functoriality.
However, there is no analogue of \Cref{thm:I} available, and Ayoub was able to construct the $!$-functoriality only for morphisms that arise as the analytification of algebraic morphisms (that is, those coming from $\Sch{}$).\footnote{In fact, he obtained this as an application of a version of \Cref{thm:I}.}
The original goal of~\cite{AGV} was to remedy this.
\end{rmk}

\begin{rmk}
Let us explain why an analogue of \Cref{thm:I} is not available and in fact might not be expected.
Indeed, in following the strategy of \Cref{sec:!-cs} one encounters the following problems in the rigid-analytic world:
\begin{enumerate}[(a)]
\item The analogue of \Cref{exc:pbc-supp} does not hold (apriori), that is, proper base change does not imply the support property.
The underlying reason is that while \bref{ax:loc} holds for $\RigSH$, it is of limited use since the complement of an open immersion is not typically a rigid-analytic space.
\item At several places in \Cref{sec:!-cs} we used Chow's lemma to reduce questions about proper morphisms to projective ones.
However, an analogue of Chow's lemma is not available in rigid-analytic geometry, thus making this strategy infeasible.
\end{enumerate}

\begin{rmk}
On the other hand, morphisms locally of finite type between rigid-analytic spaces are still \emph{weakly} compactifiable, at least locally.
More precisely, every $f:X\to Y$ locally of finite type is, locally on~$X$, the composition of a locally closed immersion followed by a proper morphism.
Therefore, once one knows \bref{supp} and the existence of right adjoints to proper push-forwards one can then follow essentially the same strategy in constructing the exceptional functoriality as in \Cref{sec:functoriality}.
The existence of the required right adjoints follows easily from the fact that the \icats $\RigSH(X)$ are compactly generated and that the inverse image functors along proper morphisms preserve compact objects.
\end{rmk}
In the remainder of this section we will sketch how to prove \bref{supp}.
\end{rmk}

\begin{rmk}
\label{rmk:raynaud}
The proof still employs a transfer from algebraic to rigid-analytic geometry albeit in a very different way.
It is based on Raynaud's approach to rigid-analytic geometry that can be roughly described by the following picture:
\begin{equation}
\label{eq:raynaud}
\begin{tikzcd}
&\FSch{}
\ar[ld, "\sigma" {swap,yshift=.3ex}]
\ar[rd, "\rho"]
\\
\Sch{}
&&
\RigSpc{}
\end{tikzcd}
\end{equation}
Here, formal schemes sit at the top and admit two functors: the `special fiber' that associates to $\cX$ its underlying topological space with the reduced scheme structure $\sigma(X)$, and the 'generic fiber' $\rho$ that is a categorical localization.
More precisely, the category $\RigSpc{}$ is, as a first approximation, the localization of $\FSch{}$ with respect to so-called `admissible blow-ups': blow-ups with center `contained in the special fiber'.
This approximation becomes correct if one imposes finiteness conditions on the formal schemes involved (adic with finitely-generated ideals of definition) and if one allows rigid-analytic spaces to be glued along open immersions.
\end{rmk}

\begin{rmk}
Passing to stable motivic homotopy theory in the three contexts in parallel gives rise to a roof like so,
\[
\begin{tikzcd}
&\FSH
\ar[ld, "\sigma^*" above, shift right=.3em, name=L]
\ar[rd, "\rho^*", shift left=.3em]
\\
\SH
\ar[ru, "\sigma_*" below, shift right=.3em, name=R, "\sim" above]
\ar[rr, "\rho^*" above, dotted, shift left=.3em]
&&
\RigSH
\ar[ul, "\rho_*" below, shift left=.3em]
\ar[ll, "\rho_*", shift left=.3em, dotted]
\end{tikzcd}
\]
where the components of the natural transformations $(-)^*$ are symmetric monoidal functors with right adjoints $(-)_*$, and where $\sigma^*\dashv\sigma_*$ is an adjoint equivalence, by Localization for~$\FSH$.
We continue to denote by $\rho^*$ \resp{$\rho_*$} the functors at the bottom that make the triangle commute.

The basic idea in proving \bref{supp} for $\RigSH$ is to apply $\rho_*$ to the morphism $(j_2)_\sharp(r_2)_*\to(r_1)_*(j_1)_\sharp$ and use \bref{supp} for $\SH$ to show it is an equivalence.
This requires two inputs:
\begin{enumerate}[(a)]
\item The functor $\rho_*$ needs to be sufficiently conservative.
While it isn't on the nose, it is still true (and easy to prove) that the family
\[
\left(\RigSH(S)\xto{f^*}\RigSH(X)\xto{\rho_*^{\cX}}\SH(\sigma(\cX))\right)_{f,\cX}
\]
is jointly conservative, where $f:X\to S$ runs through smooth morphisms of rigid-analytic spaces and $\cX$ is a chosen formal model of~$X$.
\item It is clear that $f^*$ \resp{$\rho_*$} commutes with the $(j_i)_\sharp$ and the $(p_i)_*$ \resp{with the $(p_i)_*$} so it remains to prove that $\rho_*$ commutes with the $(j_i)_\sharp$.
\end{enumerate}
This last point turns out to be quite involved and required a systematic study of $\RigSH$.
We will not go into the details here and refer to~\cite[Theorem~4.1.3]{AGV} instead.
On the other hand, this systematic study leads to other results of independent interest which we do want to mention.
\end{rmk}

\begin{thrm}[{\cite[Theorem~3.3.3]{AGV}}]
\label{thm:agv-main}
\begin{enumerate}
\item The components of the natural transformation
$\SH(\sigma(-),\rho_*\one)\to\RigSH(\rho(-))$ are fully faithful.
\item The natural transformation $\SH^{\textup{\'et}}(\sigma(-),\rho_*\Q)\to\RigSH^{\textup{\'et}}(\rho(-),\Q)$ exhibits the latter as the rig-\'etale sheafification of the former.
\end{enumerate}
\end{thrm}
Here, the natural transformations in the statement are between $\PrLst$-valued functors on $\RigSpc{\textup{op}}$ (viewed as having the same objects as $\FSch{}$, see \Cref{rmk:raynaud}).
The notation $\SH(X,A)$ is a shorthand for the \icat of $A$-modules, $A$ being a commutative algebra object in~$\SH(X)$, as in \Cref{sec:motiv-cosy}.
The first part of \Cref{thm:agv-main} can be read as saying that a whole chunk of $\RigSH$ (one might call it the part of good reduction) admits a completely algebraic description.
In fact, in good cases the commutative algebra $\rho_*\one$ can be computed.
For example, over the $p$-adic integers $\rho_*^{\Z_p}\one\simeq\Hm^\bullet(\mathbb{G}_{\textup{m}})$ and we deduce that
\[
\SH^{\textup{up}}(\FF_p)\isoto\RigSH^{\textup{gr}}(\Q_p)
\]
where the domain denotes the \emph{unipotent motives}, that is, the localizing \subicat of $\SH(\mathbb{G}_{\textup{m},\FF_p})$ generated by the constant motives.

Finally, the second part of \Cref{thm:agv-main} gives a precise measure of the failure of all rigid-analytic motives to be of good reduction.
In comparison to the first part, some additional hypotheses are necessary, for example \'etale-(hyper)sheafification and $\Q$-linearity.
All in all, \Cref{thm:agv-main} is a vast generalization of~\cite[Scholie~1.3.26.(1)]{ayoub:rigid-motives} which inspired the strategy in the first place.

\phantomsection
\addcontentsline{toc}{section}{References}
\bibliographystyle{alpha}
\bibliography{ref}

\end{document}